\documentclass[12pt]{amsart}

\usepackage{amssymb}
\usepackage{CJK}
\usepackage[all]{xy}
\usepackage{dsfont}

\SelectTips{cm}{}
\calclayout

\numberwithin{equation}{section}

\theoremstyle{plain}
\newtheorem{theorem}[equation]{Theorem}
\newtheorem{proposition}[equation]{Proposition}
\newtheorem{lemma}[equation]{Lemma}
\newtheorem{corollary}[equation]{Corollary}
\newcommand{\Hom}{\operatorname{Hom}}
\theoremstyle{definition}
\newtheorem{definition}[equation]{Definition}
\newtheorem{example}[equation]{Example}

\theoremstyle{remark}
\newtheorem{remark}[equation]{Remark}

\newcommand{\mc}{\operatorname{\mathsf{MC}}}

\newcommand{\ul}{\underline}

\newcommand{\To}{\longrightarrow}

\def\mcC{\mathcal{C}}

\def\Ho{\mathsf{Ho}}
\def\mcD{\mathcal{D}}

\def\A{\mathcal{A}}

\def\X{\mathcal{X}}

\begin{document}
\title [\tiny{The left (right) triangulated structures of stable categories
}]{The left and right triangulated structures of stable categories}
\author [\tiny{Zhi-Wei Li}] {Zhi-Wei Li}
\thanks{\tiny{Supported by JSNU12XLR025.}}
\thanks{\tiny{E-mail address: zhiweili$\symbol{64}$jsnu.edu.cn }} \maketitle
\begin{center}
\tiny{School of Mathematics and Statistics, \ \ Jiangsu Normal University\\
Xuzhou 221116, Jiangsu, P. R. China}
\end{center}
\vskip 5pt

\begin{abstract} Beligiannis and Marmaridis [\emph{Comm. in Algebra,} 22(12)
(1994), 5021-5036] constructed the left and right triangulated structures on the stable categories of additive categories induced from some homological finite subcategories. We extend their results to slightly more general settings. As an application of our results we give some new examples of stable categories which have left or right triangulated structures from abelian model categories. An interesting outcome is that we can describe the pretriangulated structures of the homotopy categories of abelian model categories via the ones of stable categories.

\end{abstract}

\maketitle

\setcounter{tocdepth}{1}
\tableofcontents

\section{Introduction}

In \cite{Verdier} J. -L. Verdier introduced the notion of triangulated structures on an additive category equipped with an additive autoequivalent functor.  Verdier's triangulated category can be viewed as ``stable" in the sense that the endofunctor is an equivalence. Almost at the same time, in \cite[Theorem I.2.2, Proposition I.3.5]{Quillen67} D. Quillen proved that the homotopy categories of his pointed model categories carried some analogues triangulated structures which satisfied Verdier's triangulated axioms except that the corresponding endofunctors are not equivalences. This can be seen as the first step for studying the ``unstable" triangulated categories in the sense that the endofunctor needs not to be an equivalence. Later, K. S. Brown generalized Quillen's result to the homotopy categories of ``categories of fibrant objects" \cite[Theorem 3]{Brown74}. For additive categories, such ``unstable" triangulated structures also have been studied by many authors \cite{Keller/Vossieck87}, \cite{Beligiannis/Marmaridis94}, \cite{Beligiannis01} and \cite{Beligiannis/Reiten07}.

Recall that a left triangulated category is an additive category $\mcC$ equipped with an additive endofunctor $\Omega$ ( called the loop functor) and a class of left triangles which satisfies Verdier's triangulated axioms \cite[Definition 2.3]{Beligiannis/Marmaridis94}. Dually, there is a notion of right triangulated category.  The additive categories equipped with homological finite subcategories are important examples of the left (right) triangulated categories. Let $\mcC$ be an additive category and $\X$ a full additive subcategory of $\mcC$. If $\X$ is contravariantly finite in $\mcC$ and any $\X$-epic has a kernel, A. Beligiannis and N. Marmaridis proved that the stable category $\mcC/\X$ has a left triangulated category; the loop functor $\Omega$ and the class of left triangles are defined via right $\X$-approximations \cite[Theorem 2.12]{Beligiannis/Marmaridis94}. Dually, if $\X$ is covariantly finite in $\mcC$ and any $\X$-monic has a cokernel, the stable category $\mcC/\X$ becomes a right triangulated category. If $\mcC$ is an abelian category and $\X$ is a functorially finite subcategory, then the stable is a pretriangulated category which is an additive category $\mcC$ equipped with an adjoint pair $(\Omega, \Sigma)$ of endofunctors and in addition with classes of left and right triangles. The classes of left and right triangles satisfy compatible conditions \cite[Chapter II.1]{Beligiannis/Reiten07}.

Let $\A$ be a bicomplete abelian category. M. Hovey discovered a correspondence between the abelian model structures on $\A$ and the compatible complete cotorsion pairs of $\A$ \cite[Theorem 2.2]{Hovey02}. The connection between model structures and cotorsion pairs is also discussed by A. Beligiannis and I. Reiten in  \cite[Chapter VIII]{Beligiannis/Reiten07}.
By Hovey's correspondence, given an abelian model structure of $\A$, there are two compatible complete cotorsion pairs $(\A_c, \A_{f}\cap \A_{triv})$ and $(\A_{c}\cap \A_{triv}, \A_f)$, where $\A_f$ (respectively $\A_c$) is the full subcategory of fibrant (respectively cofibrant) objects and $\A_{triv}$ the full subcategory of trivial objects of $\A$. In particular, $\A_c, \A_f$ are additive subcategories of $\A$. Let $\omega$ be the full subcategory of trivial bifibrant objects of $\A$. Then $\omega$ is contravariantly finite in $\A_f$ and covariantly finite in $\A_c$. Especially, $\omega$ is functorially finite in $\A_{cf}=\A_c\cap \A_f$ and the homotopy category $\Ho(\A)$ of the abelian model category $\A$ is equivalent to the stable category $\A_{cf}/\omega$ \cite[Theorem VII.4.2]{Beligiannis/Reiten07}, \cite[Proposition 4.3,4.7]{Gillespie11}, \cite[Proposition 1.1.14]{Becker12}. So by D. Quillen's theorem \cite[Theorem I.2.2, Proposition I.3.5]{Quillen67}, $\A_{cf}/\omega$ has a pretriangulated structure in the sense of \cite[Definition 4.9]{Beligiannis01} induced from $\Ho(\A)$. Unfortunately, in general, it is difficult to describe explicitly the pretriangulated structure of $\Ho(\A)$ and $\A_{cf}/\omega$. We ask if it is possible to realize the stable categories $\A_f/\omega$ as a left triangulated category and $\A_c/\omega$ as a right triangulated category, and if these left and right triangulated structures induce the pretriangualted structure of the homotopy category $\Ho(\A)$ and the stable category $\A_{cf}/\omega$.  Now $\omega$-epics doesn't necessarily have kernels in $\A_f$ and $\omega$-monics doesn't necessarily have cokernels in $\A_c$. Also, we don't know how to realize the stable categories as the homotopy categories of some model categories.  So both Beligiannis-Marmaridis' s theorem \cite[Theorem 2.12]{Beligiannis/Marmaridis94} and Quillen's result \cite[Theorem I.2.2, Proposition I.3.5]{Quillen67} fail in this case. We observed that in such case, special $\omega$-epics (see Definition 3.3) have kernels in $\A_f$ and special $\omega$-monics (see the dual of Definition 3.3) have cokernels in $\A_c$. These motivate us to give the following theorem (Theorem 3.6):

\vskip10pt
\noindent{\bf Main Theorem 1: } \ {\it Let $\mcC $ be an additive category and $\X$ a full additive subcategory of $\mcC$.
\vskip5pt
$(1)$ \  If $\X$ is contravariantly finite in $\mcC$ and any special $\X$-epic has a kernel in $\mcC$, then the stable category $\mcC/\X$ has a left triangulated structure induced by $\X$.

$(2)$ \ If $\X$ is covariantly finite in $\mcC$ and any special $\X$-monic has a cokernel in $\mcC$, then the stable category $\mcC/\omega$ has a right triangulated structure induced by $\X$.}

\vskip10pt
The proof of this theorem is along the idea of Beligiannis-Marmaridis \cite{Beligiannis/Marmaridis94}.  Beliannis-Marmaridis's theorem \cite[Theorem 2.12]{Beligiannis/Marmaridis94} is a special case of our Main Theorem 1. As an application, for a bicomplete abelian model category $\A$,  we see that the stable category $\A_f/\omega$ has a left triangulated structure induced by $\omega$ and $\A_c/\omega$ has a right triangulated structure induced by $\omega$. So by our Main Theorem 1, we can get more examples of ``unstable" triangulated categories. In particular,  for an abelian model category $\A$, using the left (respectively right) triangulated structure of the stable category $\A_f/\omega$ (respectively $\A_c/\omega$) we can describe the left (respectively right) triangulated structure of the homotopy category $\Ho(\A_f)$ (respectively $\Ho(\A_c)$) (See Theorem 4.12):

 \vskip10pt
\noindent{\bf Main Theorem 2: } \ {\it Let $\A $ be an abelian model category and $\omega$ the full subcategory of trivial bifibrant objects of $\A$.
\vskip5pt
$(1)$ \  The left triangulated structure of $\A_f/\omega$ induces a left triangulated structure of $\Ho(\A_f)$.

$(2)$ \ The right triangulated structure of $\A_c/\omega$ induces a right triangulated structure of $\Ho(\A_c)$.}
\vskip10pt
These left and right triangulated structures induce the pretriangulated structures of the homotopy category $\Ho(\A)$ and the stable category $\A_{cf}/\omega$ (see Corollary 4.14).

The paper is organized as follows. In Section 2, we recall the constructions of loop and suspension functors on the stable category of an additive category induced by homological finite subcategories in \cite{Beligiannis/Marmaridis94}. In Section 3.1 and 3.2, we recall the notion of the left (right) triangulated category in \cite{Beligiannis/Marmaridis94} and give the notions of a special $\X$-epic and a special $\X$-monic. In Section 3.3, after the distinguished left (right) triangles are defined on the stable category $\mcC/\X$, we prove our Main Theorem 1. In Section 4, we construct some examples which satisfy the assumptions of our main theorem from abelian model categories and prove our Main Theorem 2. In Section 5, the appendix, we show that if the stable category of an additive category is induced by some homological finite additive subcategory which is closed under direct summands, then the corresponding stable category can be realized as the homotopy category of a fibration category or cofibration category.

\vskip10pt

\section{The loop and suspension functors on the stable categories}

\vskip10pt

Let $\mcC$ be an additive category and $\X$ an additive subcategory of $\mcC$. In this section, we recall the constructions of the loop (respectively suspension) functors under the assumption that $\X$ is contravariantly (respectively covariantly) finite in $\mcC$ and any right (respectively left) $\X$-approximation has a kernel (respectively cokernel)\cite{Beligiannis/Marmaridis94}. We begin by giving some basic definitions and notations of homological finite subcategories in \cite{Auslander/Smal80}.
\vskip10pt
\subsection{Homological finite subcategories and stable categories} Let $\mcC$ be an additive category and $\X$ an additive subcategory. By a {\it right $\X$-approximation} of an object $C$ in $\mcC$ we mean a morphism $p_{_C}: X_C\to C$ with $X_C\in \X$ such that the induced map $\Hom_\mcC(X, X_C)\to \Hom_\mcC(X, C)$ is surjective for all $X\in \X$. The subcategory $\X$ is said to be {\it contravariantly finite} in $\mcC$ if each object $C$ of $\mcC$ has a right $\X$-approximation.

Dually, one can define  {\it left $\X$-approximation} of an object $C$ in $\mcC$ and $\X$ is called {\it covariantly finite} in $\mcC$ if any object of $\mcC$ has a left $\X$-approximation. We say that $\X$ is {\it functorially finite} in $\mcC$ if it is both contravariantly and covariantly in $\mcC$.

Given morphisms $f,g: C\to D$ in $\mcC$, we say that $f$ is {\it stably equivalent } to $g$, written $f\sim g$, if $f-g$ factors through some object of $\X$. It is well known that stable equivalence is an equivalence relation which is compatible with composition.

Let $\mcC/\X$ be the stable category. That is a category whose objects are the same as $\mcC$ and whose morphisms are stable equivalence classes of $\mcC$. Let $\pi: \mcC\to \mcC/\X$ be the canonical quotient functor. The image $\pi(C)$ of $C\in \mcC$ is denoted by $\underline{C}$ and the image of $\pi(f)$ of any morphism $f$ is denoted by $\underline{f}$.

\vskip10pt
\subsection{The loop and suspension functors on the stable categories}
Let $\X$ be a contraviantly fintie subcategory in an additive category $\mcC$. Suppose that any right $\X$-approximation has a kernel in $\mcC$. For each object $C$ of $\mcC$, assign a right $\X$-approximation $p_{_C}: X_C\to C$ with kernel $K_C$. Given a morphism $f: C\to D$ in $\mcC$, take right $\X$-approximations $p_{_C}$ of $C$ and $p_{_D}$ of $D$ respectively. Then there is a commutative diagram with exact rows:
\[
\xymatrix{
0\ar[r] & K_C\ar[r]^{\iota_C}\ar[d]_{\kappa_f} &X_C\ar[d]_{x_f} \ar[r]^{p_{_{C}}} & C\ar[d]_f ^{\quad (*)} \\
0 \ar[r] & K_D\ar[r]^{\iota_D} & X_D \ar[r]^{p_{_{D}}} & D.}
\]

With above notations, define a functor $\Omega: \mcC\to \mcC/\X$ as follows: $\Omega$ maps each object $C\in \mcC$ to $\ul{K}_C$ and each morphism $f:C\to D$ to $\ul{\kappa}_f$. This functor is well-defined. In fact, assume that there is another $x'_f$ such that $p_{_D}x'_f=fp_{_C}$, then $p_{_D}(x_f-x'_f)=0$. Thus there exists a morphism $l: X_C\to K_D$ such that $x_f-x'_f=\iota_D l$ since $K_D=\ker p_{_D}$. Since $x_f\iota_C=\iota_D \kappa_f$ and $x'_f\iota_C=\iota_D \kappa'_f$, we have $\iota_D(\kappa_f-\kappa'_f)=(x_f-x'_f)\iota_C=\iota_D l\iota_C$. Then $\kappa_f-\kappa'_f=l\iota_C$ since $\iota_D$ is a monomorphism. That is to say $\ul{\kappa}_f=\ul{\kappa}'_f$.

For two morphisms $f,g: C\to D$, if $\ul{f}=\ul{g}$, by definition, $f-g$ factors as a composite $C\stackrel{t}\to X\stackrel{s}\to D $ with $X\in \X$. Since $p_{_D}$ is a right $\X$-approximation, there exists a morphism $l:X\to X_D$ such that $p_{_D}l=s$. Then $p_{_D}(x_f-x_g)=(f-g)p_{_C}=stp_{_C}=p_{_D}ltp_{_C}$ and thus $p_{_D}(x_f-x_g-ltp_{_C})=0$. Similarly to the above discussion, at last we see that $\ul{\kappa}_f=\ul{\kappa}_g$. So the functor $\Omega:\mcC\to \mcC/\X$ induces a functor $\mcC/\X\to \mcC/\X$ which sends each object $\ul{C}$ to $\Omega(C)$ and each morphism $\ul{f}$ to $\Omega(f)$. We call this functor a {\it loop functor} and still denote it by $\Omega$.

Dually, suppose $\X$ is a covariantly finite subcategory of $\mcC$ and any left $\X$-approximation has a cokernel in $\mcC$. Assign each object $C\in \mcC$ a left $\X$-approximation $C\stackrel{i_C}\to X^C$ with cokernel $ K^C $, we can define a functor $\Sigma : \mcC\to \mcC/\X$ which sends object $C$ to $\ul{K}^C$. We call the induced functor on $\mcC/\X$ a {\it suspension functor} and still denote it by $\Sigma$.
\vskip10pt

\section{The left (right) triangle structures on stable categories.}

This section is devoted to giving a left (respectively right) triangle structure on the stable category $\mcC/\X$ under the assumption that every special $\X$-epic (respectively $\X$-monic) morphism has a kernel (respectively cokernel) when $\X$ is a contravariantly (respectively covariantly) finite subcategory in $\mcC$. For simplicity, we only concentrate on the discussion of the left triangulated structure on $\mcC/\X$.  We begin by recalling the definition of a left (respectively, right) triangulated category in \cite{Beligiannis/Marmaridis94}.

\subsection{The left and right triangulated categories} \ Let $\mcC$ be an additive category and $\Omega$ an additive covariant endofunctor on $\mcC$. Let $\triangle$ be a class of {\it left} triangles of the form $\Omega(A)\stackrel{h}\to C\stackrel{g}\to B\stackrel{f}\to A$. Morphisms between left triangles are triples $(\gamma, \beta,\alpha)$ of morphisms which make the following diagram commutate
\[
\xymatrix{
\Omega(A)\ar[r]^h\ar[d]_{\Omega(\alpha)} & C\ar[r]^{g}\ar[d]_{\gamma} & B\ar[d]_{\beta} \ar[r]^{f} & A\ar[d]_\alpha  \\
\Omega(A') \ar[r]^{h'} & C'\ar[r]^{g'} & B' \ar[r]^{f'} & A'.}
\]

\begin{definition} \ The pair $(\Omega, \triangle)$ is called a {\it left triangulated structure} on $\mcC$  if $\triangle$ is closed under isomorphisms and satisfies the following four axioms:
\vskip5pt
${\rm(LT1)}$ \ For any morphism $f: B\to A$ there is a left triangle in $\triangle$ of the form $\Omega(A)\to C\to B\stackrel{f}\to A$. For any object $A\in \mcC$, the left triangle $0\to A\stackrel{1_A}\to A\to 0$ is in $\triangle$.
\vskip5pt
${\rm (LT2)}$  (Rotation axiom)\ For any left triangle $\Omega(A)\stackrel{h}\to C\stackrel{g}\to B\stackrel{f}\to A$ in $\triangle$, the left triangle $\Omega{B}\stackrel{-\Omega(f)}\to\Omega(A)\stackrel{h}\to C \stackrel{g}\to B$ is also in $\triangle$.
\vskip5pt
${\rm (LT3)}$ \ For any two left triangles $\Omega(A)\stackrel{h}\to C\stackrel{g}\to B\stackrel{f}\to A,\Omega(A')\stackrel{h'}\to C'\stackrel{g'}\to B'\stackrel{f'}\to A'$ in $\triangle$ and any two morphisms $\alpha:A\to A', \beta:B\to B'$ of $\mcC$ with $\alpha f=f'\beta$, there is a morphism $\gamma: C\to C'$ of $\mcC$ such that $(\gamma, \beta, \alpha)$ is a morphism from the first left triangle to the second.

\vskip5pt
${\rm (LT4)}$ (Octahedral axiom)\ For any two morphisms $g:C\to B $ and $f:B\to A, $ let $\Omega(B)\stackrel{m}\to D\stackrel{k}\to C\stackrel{g}\to B$, $\Omega(A)\stackrel{n}\to E\stackrel{h}\to C\stackrel{fg}\to A$,$\Omega(A)\stackrel{i}\to F\stackrel{l}\to B\stackrel{f}\to A$ be the left triangles in $\triangle$ corresponding to $g, fg$ and $f$ respectively. There is a commutative diagram
\[
\xymatrix{
& \Omega(F)\ar[d]^{m\Omega(l)} &  & \\
\Omega(B)\ar[r]^m\ar[d]_{\Omega(f)} & D\ar[r]^{k}\ar[d]^{\alpha} & C\ar@{=}[d] \ar[r]^{g} & B\ar[d]_f  \\
\Omega(A) \ar[r]^{n}\ar@{=}[d] & E\ar[r]^{h}\ar[d]^\beta & C \ar[r]^{fg} \ar[d]^{g}& A\ar@{=}[d]\\
\Omega(A)\ar[r]^{i} & F\ar[r]^{l}& B\ar[r]^{f}& A}
\]
where the second column from the left is a left triangle in $\triangle$.
\end{definition}
\vskip10pt
Dually, we can define the {\it right triangulated structure} $(\Sigma, \triangle')$ on $\mcC$, where $\Sigma$ is a covariant additive endofunctor of $\mcC$ and $\triangle'$ is a class of right triangles of the form $A\to B\to C\to \Sigma(A)$ which satisfies the dual right triangulated axioms.

An additive category $\mcC$ is called a {\it left triangulated category } if there is a left triangulated structure on $\mcC$. Dually, an additive category $\mcC$ is called a {\it right triangulated category } if there is a right triangulated structure on it.

\begin{remark} \ If the endofunctor $\Omega$ (respectively, $\Sigma$) is an autoequivalence, the left (respectively, right) triangulated category $(\mcC, \Omega, \triangle)$ (respectively, $(\mcC, \Sigma, \triangle')$) is a triangulated category in the sense of Verdier in \cite{Verdier}.
\end{remark}

\subsection{Special $\X$-epic morphisms} \  Let $\X$ be a contravariantly finite subcategory of an additive category $\mcC$. Recall that a morphism $f: B\to A$ in $\mcC$ is called an {\it $\X$-epic} if for any $X\in \X$ the induced map $\Hom_\mcC(X, g): \Hom_{\mcC}(X, B)\to \Hom_{\mcC}(X,A)$ is surjective.

\begin{definition} \ A morphism $f: B\to A$ in $\mcC$ is called a {\it special  $\X$-epic }if it is of the following form: $$B\oplus X_A\stackrel{(g, p_A)}\To A$$
where $g$ is a morphism of $\mcC$ and $p_{_A}$ is a right $\X$-approximation of $A$.
 \end{definition}

 By definition, a right $\X$-apprixomation is a special $\X$-epic and  a special $\X$-epic is an $\X$-epic.
\vskip10pt

Dually, if $\X$ is covariantly finite in $\mcC$, we have the notions of an {\it $\X$-monic } and a {\it special $\X$-monic}.

\vskip10pt
Let $\mcC$ be an additive category and $\X$ an additive subcategory of $\mcC$. Assume that $\X$ is contravariantly finite in $\mcC$ and any special $\X$-epic has a kernel. Then the pullback of any morphism $f:B\to A$ along the right $\X$-approximation $p_{_A}:X_A\to A$ exists and there is a commutative diagram
\[
\xymatrix{
0  \ar[r] & K_A \ar[r]^{\zeta_{_f} \ \ }\ar@{=}[d] & {\rm PB}(f)\ar[d]_{\theta_{_f}} \ar[r]^{~~~\eta_{_f}} & B \ar[d]_f ^{\quad(**)}\\
0  \ar[r] & K_A \ar[r]^{\iota_A} & X_A \ar[r]^{p_{_{A}}} & A }
\]
with exact rows and right square being a pullback. This is since $(f, p_{_A}): B\oplus W_A\to A$ is a special $\X$-epic and thus it has a kernel by assumption. We denote its kernel by ${\rm PB}(f)$. Furthermore, in the above diagram, $\eta_{_f}$ is $\X$-epic by the universal property of pullback and the $\X$-epic property of $p_{_A}$. In fact the assumption that any special $\X$-epic has a kernel is equivalent to the existence of the pullback for any form of the diagram
 \[
\xymatrix{
   & B \ar[d]_f \\
  X_A \ar[r]^{p_{_{A}}} & A }
\]
with $p_{_A}$ being a right $\X$-approximation and $f$ a morphism in $\mcC$.

\vskip10pt
If $g: C\to B$ is an $\X$-epic with a kernel, then we have a commutative diagram with exact rows in $\mcC:$
\[
\xymatrix{
0 \ar[r]& K_B\ar[r]^{\iota_B}\ar[d]_{\gamma_g} & X_B\ar[d]_{\delta_g} \ar[r]^{p_{_{B}}} & B\ar@{=}[d]^{\quad (***)} \\
0 \ar[r] & \ker g \ar[r]^{\iota_g} & C \ar[r]^{g} & B }
\]
where the existence of $\delta_g$ is since $g$ is $\X$-epic and $X_B\in \X$. Furthermore, it is not difficult to see that $\gamma_g$ is independently on the choice of $\delta_g$ in the stable category $\mcC/\X$.

\vskip10pt
The following lemma is the tool for us to prove the octahedral axiom for the left triangulated structure in the stable category $\mcC/\X$ in the next section. Compare Lemma 2.11 of \cite{Beligiannis/Marmaridis94}.

 \begin{proposition} \ Let $\mcC$ be an additive category and $\X$ a contravariantly finite subcategory in $\mcC$. Assume that any special $\X$-epic has a kernel in $\mcC$.  Let $g: C\to B , f:B\to A$ be two morphisms and $g$ is an $\X$-epic with a kernel. Then
 \vskip 5pt
 $(i)$ there is a commutative diagram in $\mcC$ except the left above square:
\[
\xymatrix{ & K_{{\rm PB}(f)}\ar[d]_{-\gamma_{\beta}}& & \\
K_{B} \ar[d]^{\kappa_{f}}\ar[r]^{-\gamma_g \ } & \ker g \ar[r]^{\iota_g} \ar[d]^\alpha & C \ar[r]^{g} \ar[d]^{\left(\begin{smallmatrix}
1 \\
0
\end{smallmatrix}\right)} & B\ar[d]^{f} \\
K_A \ar@{=}[d]\ar[r]^{\zeta_{fg} \ } & {\rm PB}(fg) \ar[r]^{\left(\begin{smallmatrix}
\eta_{fg} \\
-\theta_{fg}
\end{smallmatrix}\right)}\ar[d]^{\beta}& C\oplus X_A\ar[r]^{ \ \ \ (fg, p_{_{A}})}\ar[d]^{\left(\begin{smallmatrix}
g& 0 \\
0 & 1
\end{smallmatrix}\right)} & A\ar@{=}[d]\\
K_A \ar[r]^{\zeta_f \ } & {\rm PB}(f)\ar[r]^{\left(\begin{smallmatrix}
\eta_{f} \\
-\theta_{f}
\end{smallmatrix}\right)}  & B\oplus X_A \ar[r]^{ \ \ \ (f, p_{_A})}  & A }
\]
where $\beta$ is $\X$-epic with kernel $\alpha$;
\vskip5pt
$(ii)$ \ $-\ul{\alpha} \ul{\gamma}_g=\ul{\zeta}_{fg} \Omega(f), \ul{\gamma}_{\beta}=\ul{\gamma}_g\Omega(\eta_f)$ in the stable category $\mcC/\X$, where $\Omega:\mcC\to \mcC$ is the loop functor defined in Section 2.2.
\end{proposition}

\begin{proof} $(i)$ \ The kernels  $\ker g, {\rm PB}(fg) \ \mbox{and} \ {\rm PB}(f)$  exist by assumption. The morphisms $\kappa_f, p_{_A}$ are as defined in $(*)$ in Section 2.2 and $\eta_{_f}, \theta_{f}, \zeta_{_f}, \eta_{fg},$
$\theta_{fg}, \zeta_{{fg}}, \gamma_g$ are as defined in $(**),(***)$ in Section 3.2. By the universal property of kernels there are morphisms $\alpha, \beta$ such that the middle squares are commutative. The right squares are obviously commutative.
Note that $\left(\begin{smallmatrix}
g& 0 \\
0 & 1
\end{smallmatrix}\right)$ is $\X$-epic, we know that $\beta$ is also $\omega$-epic because of the following commutative diagram with exact rows:
\[
\xymatrix{
0  \ar[r] & (\X,{\rm PB}(fg)) \ar[r]^{{\left(\begin{smallmatrix}
\eta_{fg} \\
-\theta_{fg}
\end{smallmatrix}\right)_*} }\ar[d]^{\beta_*} & (\X, C\oplus X_A)\ar[d]^{\left(\begin{smallmatrix}
g& 0 \\
0 & 1
\end{smallmatrix}\right)_*} \ar[r]^{\ \ \ \ (fg, p_{_A})_*} & (\X,A) \ar@{=}[d]\ar[r] & 0 \\
0  \ar[r] &(\X, {\rm PB}(f)) \ar[r]_{\left(\begin{smallmatrix}
\eta_{f} \\
-\theta_{f}
\end{smallmatrix}\right)_*} & (\X, B\oplus X_A) \ar[r]_{\ \ \ \ (f, p_{_A})_*} &(\X, A) \ar[r]& 0 }
\]
where $(\X,-)$ denotes $\Hom_\mcC(X,-)$ for any $X\in \X$.

 Since $\left(\begin{smallmatrix}
\eta_{f} \\
-\theta_{f}
\end{smallmatrix}\right)\beta \alpha=\left(\begin{smallmatrix}
g& 0 \\
0 & 1
\end{smallmatrix}\right)\left(\begin{smallmatrix}
1 \\
0
\end{smallmatrix}\right)\iota_g=\left(\begin{smallmatrix}
g\iota_g \\
0
\end{smallmatrix}\right)=0$ and $\left(\begin{smallmatrix}
\eta_{f} \\
-\theta_{f}
\end{smallmatrix}\right)$ is a monomorphism, we know that $\beta \alpha=0$. Assume that there is a morphism $h: D\to {\rm PB}(fg)$ such that $\beta h=0$. Then $0=\left(\begin{smallmatrix}
\eta_{f} \\
-\theta_{f}
\end{smallmatrix}\right)\beta h=\left(\begin{smallmatrix}
g& 0 \\
0 & 1
\end{smallmatrix}\right)\left(\begin{smallmatrix}
\eta_{fg} \\
-\theta_{fg}
\end{smallmatrix}\right)h=\left(\begin{smallmatrix}
g\eta_{fg}h \\
-\theta_{fg}h
\end{smallmatrix}\right)=0$, i.e. $g(\eta_{fg}h)=0$ and $-\theta_{fg}h=0$. Thus there is a morphism $l:D\to \ker g$ such that $\iota_g l=\eta_{fg}h$. So we have $\left(\begin{smallmatrix}
\eta_{fg} \\
-\theta_{fg}
\end{smallmatrix}\right)\alpha l=\left(\begin{smallmatrix}
1\\
0
\end{smallmatrix}\right)\iota_g l=\left(\begin{smallmatrix}
\eta_{fg} \\
-\theta_{fg}
\end{smallmatrix}\right)h$. Since $\left(\begin{smallmatrix}
\eta_{fg} \\
-\theta_{fg}
\end{smallmatrix}\right)$ is a monomorphism, we have $\alpha l=h$, i.e. $\alpha$ is the kernel of $\beta$.

By the commutative diagram $(**)$ in Section 3.2, we have $\theta_f\zeta_f=\iota_A=\theta_{fg}\zeta_{fg}$ and then $\left(\begin{smallmatrix}
\eta_{f} \\
-\theta_{f}
\end{smallmatrix}\right)(\zeta_f-\beta\zeta_{fg})=\left(\begin{smallmatrix}
-\eta_{f}\beta \zeta_{fg} \\
-\theta_{f}\zeta_f+\theta_f\beta\zeta_{fg}
\end{smallmatrix}\right)=\left(\begin{smallmatrix}
-g\eta_{fg} \zeta_{fg} \\
-\iota_A+\theta_{fg}\zeta_{fg}
\end{smallmatrix}\right)=\left(\begin{smallmatrix}
0 \\
\iota_A-\iota_A
\end{smallmatrix}\right)=0$. Since $\left(\begin{smallmatrix}
\eta_{f} \\
-\theta_{f}
\end{smallmatrix}\right)$ is a monomorphism we know that $\zeta_f=\beta\zeta_{fg}$.

\vskip5pt
$(ii)$ \ By the commutative diagrams $(*),(**), (***)$, we have

\noindent$\left(\begin{smallmatrix}
\eta_{fg}  \\
-\theta_{fg}
\end{smallmatrix}\right)(\zeta_{fg} \kappa_f+\alpha \gamma_g)=\left(\begin{smallmatrix}
\iota_g\gamma_g \\
-\theta_{fg}\zeta_{fg}\kappa_f
\end{smallmatrix}\right)=\left(\begin{smallmatrix}
\delta_g \iota_{B}  \\
-\iota_{_A} \kappa_f
\end{smallmatrix}\right)=\left(\begin{smallmatrix}
\delta_g  \\
-x_f
\end{smallmatrix}\right)\iota_{B}$. Since $(fg, p_{_A}) \left(\begin{smallmatrix}
\delta_g  \\
-x_f
\end{smallmatrix}\right)= fg\delta_g-x_fp_{_A}=fp_{_B}-fp_{_B}=0$, there is a morphism $s: X_B\to {\rm PB}(fg)$ such that $\left(\begin{smallmatrix}
\delta_g  \\
-x_f
\end{smallmatrix}\right)=\left(\begin{smallmatrix}
\eta_{fg}  \\
-\theta_{fg}
\end{smallmatrix}\right)s$. Then we have $\left(\begin{smallmatrix}
\eta_{fg}  \\
-\theta_{fg}
\end{smallmatrix}\right)(\zeta_{fg} \kappa_f+\alpha \gamma_g)=\left(\begin{smallmatrix}
\eta_{fg}  \\
-\theta_{fg}
\end{smallmatrix}\right)s\iota_{B}$ and thus $\zeta_{fg} \kappa_f+\alpha \gamma_g=s\iota_{B}$ since $\left(\begin{smallmatrix}
\eta_{fg}  \\
-\theta_{fg}
\end{smallmatrix}\right)$ is a monomorphism. That is to say $\zeta_{fg} \kappa_f+\alpha \gamma_g$ factors through an object in $\X$. So $-\ul{\zeta}_{fg}\Omega(f)=\ul{\zeta}_{fg} \ul{\kappa}_f=\ul{\alpha} \ul{\gamma}_g$ in $\mcC/\X$.

Now consider the commutative diagrams:

\[
\xymatrix{
 0\ar[r] &  K_{{\rm PB}(f)}\ar[r]^{\iota_{{\rm PB}(f)}}\ar[d]_{\gamma_{_\beta}} & X_{{\rm PB}(f)}\ar[d]_{\delta_\beta} \ar[r]^{p_{_{{\rm PB}(f)}}} & {\rm PB}(f)\ar@{=}[d] \\
 0\ar[r] &  {\rm PB}(g) \ar[r]^{\alpha} & {\rm PB}(fg) \ar[r]^{\beta} & {\rm PB}(f)}
\]
and
\[
\xymatrix{
  0\ar[r] &  K_{{\rm PB}(f)}\ar[r]^{\iota_{{\rm PB}(f)}}\ar[d]_{\kappa_{\eta_f}} & X_{{\rm PB}(f)}\ar[d]_{x_{\eta_f}} \ar[r]^{p_{_{{\rm PB}(f)}}} & {\rm PB}(f)\ar[d]^{\eta_f} \\
  0\ar[r] &  K_{B}\ar[r]^{\iota_{B}} & X_{B}\ar[r]^{p_{_B}} & B.}
\]
Since $g(\eta_{fg}\delta_\beta-\delta_gx_{\eta_f})=\eta_{f}\beta \delta_\beta -p_{_B}x_{\eta_{_f}}=\eta_{f}p_{_{\rm PB}(f)}-p_{_B}x_{\eta_{_f}}=0$, there exists a morphism $t: X_{{\rm PB}(f)}\to {\rm PB}(g)$ such that $\iota_g t=\eta_{fg}\delta_\beta-\delta_gx_{\eta_f}$. Then

\begin{align*}
\left(\begin{smallmatrix}
\eta_{fg} \\
-\theta_{_fg}
\end{smallmatrix}\right)\alpha(\gamma_\beta-\gamma_g \kappa_{\eta_f}) = & \left(\begin{smallmatrix}
\eta_{fg}\delta_\beta\iota_{{\rm PB}(f)}-\iota_g\gamma_g\kappa_{\eta_f} \\
0
\end{smallmatrix}\right) \\
=&\left(\begin{smallmatrix}
\eta_{fg}\delta_\beta\iota_{{\rm PB}(f)}-\delta_g\iota_B\kappa_{\eta_f} \\
0
\end{smallmatrix}\right) \\
=&\left(\begin{smallmatrix}
\eta_{fg}\delta_\beta\iota_{{\rm PB}(f)}-\delta_gx_{\eta_f}\iota_{{\rm PB}(f)} \\
0
\end{smallmatrix}\right)  \\
=&\left(\begin{smallmatrix}
\iota_g t \\
0
\end{smallmatrix}\right) \iota_{{\rm PB}(f)}\\
=&\left(\begin{smallmatrix}
\eta_{fg}\\
-\theta_{fg}
\end{smallmatrix}\right)\alpha t \iota_{{\rm PB}(f)}.
\end{align*}
Since $\left(\begin{smallmatrix}
\eta_{fg}\\
-\theta_{fg}
\end{smallmatrix}\right)\alpha$ is a monomorphism, we know that $\gamma_\beta-\gamma_g \kappa_{\eta_f}=t \iota_{{\rm PB}(f)}$, i.e. the morphism $\gamma_\beta-\gamma_g \kappa_{\eta_f}$ factors through one object in $\X$. So $\ul{\gamma}_{\beta}=-\ul{\gamma}_g \ul{\kappa}_{\eta_f}=\ul{\gamma}_g\Omega(\eta_f)$.

\end{proof}

\subsection{The left triangulated structure on $\mcC/\X$ }

For any morphism $f: B\to A$ in $\mcC$, there always has a commutative diagram as constructed in $(**)$:
\[
\xymatrix{
0  \ar[r] & K_A\ar[r]^{\zeta_f \ }\ar@{=}[d] & {\rm PB}(f)\ar[d]_{  \theta_f} \ar[r]^{~~~~~\eta_f} & B \ar[d]_f \\
0  \ar[r] & K_A \ar[r]^{\iota_A} & X_A \ar[r]^{p_{_{A}}} & A. }
\]
 The morphism $f$ induces a left triangle of the form $\Omega(A)\stackrel{\ul{\zeta}_f}\to \ul{\rm PB}(f)\stackrel{\ul{\eta}_f}\to \ul{B}\stackrel{\ul{f}}\to \ul{A} $ in the stable category $\mcC/\X$. The left triangles which are isomorphic to the above forms are called {\it distinguished left triangles}.

  If $f:A\to B$ is an $\X$-epic with a kernel, then we also have another commutative diagram as constructed in $(***)$:
 \[
\xymatrix{
0 \ar[r]& K_A\ar[r]^{\iota_B}\ar[d]_{\gamma_f} & X_A\ar[d]_{\delta_f} \ar[r]^{p_{_{A}}} & A\ar@{=}[d] \\
0 \ar[r] & \ker f\ar[r]^{\iota_f} & B \ar[r]^{f} & A }
\]
and $\gamma_f$ is unique in the stable category $\mcC/\X$. That is to say such morphism $f$ induces another left triangle $\Omega(A)\stackrel{-\ul{\gamma}_f}\to \ul{\ker}(f)\stackrel{\ul{\iota}_f}\to \ul{B}\stackrel{\ul{f}}\to \ul{A} $ in $\mcC/\X$. The left triangles which are isomorphic to the above forms are called {\it induced left triangles}.

\vskip10pt
The next lemma shows that the distinguished left triangles and induced left triangles are the same.

\vskip10pt
\begin{lemma} \cite[Proposition 2.10]{Beligiannis/Marmaridis94} \ Any distinguished left triangle is induced one and any induced left triangle is distinguished one.
\end{lemma}

\vskip10pt
We denote by $\triangle_\X$ the class of distinguished left triangles.
\vskip10pt
Dually, if $\X$ is covariantly finite in $\mcC$ and each special $\X$-monic has a cokernel. For a morphism $g: C\to D$, we construct a right triangle as $\ul{C}\stackrel{\ul{g}}\to \ul{D}\stackrel{\ul{\mu}^g}\to \ul{{\rm PO}}(g)\stackrel{\ul{\pi}^g}\to \Sigma(C) $  by the commutative diagram in $\mcC$
\[
\xymatrix{
C  \ar[r]^{\nu^C} \ar[d]_g & X^C\ar[r]^{\pi^C \ }\ar[d] & \Sigma(C)\ar@{=}[d] \ar[r] & 0  \\
D  \ar[r]^{\mu^g} & {\rm PO}(g) \ar[r]^{\pi^g} & \Sigma(C) \ar[r] & 0 }
\]
with exact rows and the left square being a pushout. Define the distinguished right triangles as those which are isomorphic to the ones as the above form.   The class of distinguished right triangles is denoted by $\triangle^\X$.  We have the following theorem:
\vskip10pt
\begin{theorem} \ Let $\mcC$ be an additive category and $\X$ an additive subcategory.
\vskip5pt
$(i)$ \ If $\X$ is contravariantly finite in $\mcC$ and every special $\X$-epic morphism has a kernel. Then $(\mcC/\X, \Omega, \triangle_\X)$ is a left triangulated category.

$(ii)$ \ Dually, if $\X$ is covariantly finite in $\mcC$ and every special $\X$-monic morphism has a cokernel. Then $(\mcC/\X, \Sigma, \triangle^\X)$ is a right triangulated category.
\end{theorem}

Since special $\X$-epics are $\X$-epics and special $\X$-monics are $\X$-monics, we have an immediately corollary:
\vskip10pt
\begin{corollary} $($\cite[Theorem 2.12]{Beligiannis/Marmaridis94}$)$ \   Let $\mcC$ be an additive category and $\X$ an additive subcategory.
\vskip5pt
$(i)$ \ If $\X$ is contravariantly finite in $\mcC$ and every  $\X$-epic morphism has a kernel. Then $(\mcC/\X, \Omega, \triangle_\X)$ is a left triangulated category.

$(ii)$ \ Dually, if $\X$ is covariantly finite in $\mcC$ and every $\X$-monic morphism has a cokernel. Then $(\mcC/\X, \Sigma, \triangle^\X)$ is a right triangulated category.
\end{corollary}

\vskip10pt

{\bf The proof of Theorem 3.6.} We only prove $(i)$. The proof of axioms of $({\rm LT}1)-({\rm LT}3)$ in Definition 3.1 are very similarly to the proof of Theorem 2.12 of \cite{Beligiannis/Marmaridis94}. So we only verify the octahedral axiom of Definition 3.1. For two composable morphisms $\ul{g}:\ul{C}\to \ul{B}$ and $\ul{f}:\ul{B}\to \ul{A}$ in $\mcC/\X$, without loss of generality, we may assume that $g: C\to B$ is $\X$-epic with a kernel.  The distinguished left triangle of $g$ is isomorphic to the induced left triangle of $g$ by Lemma 3.5, i.e. $\Omega(B)\stackrel{-\ul{\gamma}_g}\to \underline{\ker}(g)\stackrel{\underline{\iota}_g}\to \underline{C}\stackrel{\underline{g}}\to \underline{B} $.
 By Proposition 3.4, we obtain a commutative diagram in $\mcC/\X$ of left triangles:
\[
\xymatrix{ & \Omega({\rm PB}(f))\ar[d]_{-\ul{\gamma}_{\beta}}& & \\
\Omega(B) \ar[d]^{\Omega(f)}\ar[r]^{-\ul{\gamma}_g \ } & \ul{\ker} g \ar[r]^{\ul{\iota}_g} \ar[d]^{\ul{\alpha}} & \ul{C} \ar[r]^{\ul{g}} \ar@{=}[d] & B\ar[d]^{\ul{f}} \\
\Omega(A) \ar@{=}[d]\ar[r]^{\ul{\zeta}_{fg} \ } & \ul{{\rm PB}}(fg) \ar[r]^{
\ul{\eta}_{fg} }\ar[d]^{\ul{\beta}}& \ul{C}\ar[r]^{ \ul{fg}}\ar[d]^{\ul{g}} & \ul{A}\ar@{=}[d]\\
\Omega(A) \ar[r]^{\ul{\zeta}_f \ } & \ul{{\rm PB}}(f)\ar[r]^{\ul{\eta}_f}  & \ul{B} \ar[r]^{ \ul{f}}  & \ul{A} }
\]
where all rows are distinguished left triangles. Since $\beta$ is $\X$-epic with kernel $\alpha$, we know that the left triangle $\Omega({\rm PB}(f))\stackrel{-\ul{\gamma}_\beta}\to \ul{\ker}(g)\stackrel{\ul{\alpha}}\to \ul{{\rm PB}}(fg)\stackrel{\ul{\beta}}\to \ul{{\rm PB}}(f)$ is induced and thus is distinguished. Moreover, by $(ii)$ of Proposition 3.4, we have $\ul{\gamma}_\beta=\ul{\gamma}_g \Omega(\eta_f)$. \hfill $\Box$
\vskip10pt

\section{Stable categories induced from abelian model categories}
In this section we recall the definition of abelian model categories \cite[Definition 2.1]{Hovey02}, \cite[Definition 1.1.3]{Becker12} and Hovey's correspondence \cite[Theorem 2.2]{Hovey02}. As an application of Theorem 3.6, we will construct the left and right triangulated categories from an abelian model category. We begin by recalling the definition of model categories in the sense of Quillen  \cite{Quillen67}, \cite{Dwyer/Spalinski95},\cite{Hovey99}, \cite{Hovey02}, \cite{Hirschhorn03}, \cite{Becker12}.
\subsection{Abelian model categories}
\begin{definition}  \  A {\it model structure} on a category $\mcC$ is a triple $(\mcC of, \mathcal{W}e,$ $ \mathcal{F}ib)$ of classes of morphisms, called {\it cofibrations, weak equivalences } and {\it fibrations}, respectively, each of which contains isomorphisms, such that the following axioms are satisfied:
\vskip5pt
$\mc 1$ \ (Two out of three axiom) Given two composable morphisms $f$ and $g$ in $\mcC$, then if two of the three morphisms $f$, $g$ and $gf$ are weak equivalences, so is the third.
\vskip5pt
$\mc 2$ \  (Retract axiom)  $\mcC of, \mathcal{W}e, \mathcal{F}ib$ are closed under retracts.
\vskip5pt
$\mc 3$ \ (Lifting axiom)\ Given any commutative square
 \[
\xymatrix{
A \ar[r] \ar[d]_i& C \ar[d]^p \\
B \ar[r]_{f} \ar@{.>}[ru]& D }
\]
the dashed arrow exists, making everything commutative, provided that either $i\in \mcC of$ and $p\in \mathcal{W}e\cap \mathcal{F}ib$ or $i\in \mcC \cap \mathcal{W}e$ and $p\in \mathcal{F}ib$.
\vskip5pt
$\mc 4$ \ (Factorization axiom)  Any map $f$ can be factored in two ways: $(i)$ $f=pi$, with $i\in \mcC of, p\in \mathcal{W}e\cap \mathcal{F}ib$, and $(ii)$ $f=pi$,  with $i\in \mcC of\cap \mathcal{W}e, p\in \mathcal{F}ib$.
\end{definition}
\vskip10pt

A {\it model category} is a bicomplete category (i.e. a category possessing arbitrary small limits and colimits) equipped with a model structure.
\vskip5pt
Given a {\it pointed} model category $\mcC$ (a model category with zero object), an object $A\in \mcC$ is called {\it cofibrant} if $0\to A\in \mathcal{C} of$, it is called {\it fibrant} if $A\to 0\in \mathcal{F}ib$, and it is called {\it trivial} if $0\to A\in \mathcal{W}e$ (equivalently to $A\to 0\in \mathcal{W}e$). We use $\mcC_c, \mcC_f, \mcC_{triv}$ to denote the classes of cofibrant, fibrant, and trivial objects, respectively. The object in $\mcC_{cf}:=\mcC_c\cap \mcC_f$ is called {\it bifibrant}. A morphism which are both a cofibration (respectively fibration) and a weak equivalence is called {\it acyclic cofibration} (respectively {\it acyclic fibration}).

\begin{definition} \ A model structure on an abelian category is called {\it abelian} if the cofibrations equal monomorphisms with cofibrant cokernels and the fibraions equal epimorphisms with fibrant kernels. An {\it abelian model category} is a bicomplete abelian category equipped with an abelian model structure.
\end{definition}

\begin{remark} \
For an abelian model category $\A$, the acyclic cofibrations are precisely the monomorphisms with trivial cofibrant cokernels, the the acyclic fibrations are precisely the epimorphisms with trivial fibrant kernels \cite[Proposition 4.2]{Hovey02}.
\end{remark}

\vskip10pt
\subsection{Hovey's correspondence }
Hovey characterizes the abelian model structures in terms of complete cotorsion pairs in \cite{Hovey02}. Before we state Hovey's correspondence, we first give the related definitions of cotorsion pairs.

\begin{definition} \ \cite[Definition 2.3]{Hovey02} Let $\A$ be an abelian category. A {\it cotorsion pair} in $\A$ is a pair $(\mathcal{D}, \mathcal{E})$ of classes of objects of $\A$ such that
\vskip5pt
$(i)$ \ $\mathcal{D}=\{D\in \A \ | \ {\rm Ext}^1_\A(D, \mathcal{E})=0\}.$
\vskip5pt
$(ii)$ \ $\mathcal{E}=\{E\in \A \ | \ {\rm Ext}^1_\A(\mathcal{D}, E)=0\}.$
\end{definition}
The cotorsion pair $(\mathcal{D}, \mathcal{E})$ is called {\it complete} if the following conditions are satisfied:
\vskip5pt
$(iii)$ \ $(\mathcal{D}, \mathcal{E})$ has {\it enough projectives}, i.e. for each $A\in \A$ there exists a short exact sequence $0\to E\to D\to A\to 0$ such that $D\in \mathcal{D}$ and $E\in \mathcal{E}$.
\vskip5pt
$(iv)$ \ $(\mathcal{D}, \mathcal{E})$ has {\it enough injectives}, i.e. for each $A\in \A$ there exists a short exact sequence $0\to A\to E\to D\to 0$ such that $D\in \mathcal{D}$ and $E\in \mathcal{E}$.

\vskip10pt
For a convenient statement, we follow Gillespie \cite[Definition3.1]{Gillespie12} to define the notion of Hovey triple. Recall that a full subcategory $\mathcal{W}$ of an abelian category $\A$ is called {\it thick} if it is closed under direct summands and if two out of three of the terms in a short exact sequence are in $\mathcal{W}$, then so is the third.

\begin{definition} \ A triple $(\mathcal{D}, \mathcal{W}, \mathcal{E})$ of classes of objects in an abelian category $\A$ is called a {\it Hovey triple} if $\mathcal{W}$ is a thick subcategory of $\A$ and both $(\mathcal{D}, \mathcal{W}\cap \mathcal{E})$ and $(\mathcal{D}\cap \mathcal{W}, \mathcal{E})$ are complete cotorsion pairs in $\A$.
\end{definition}

\begin{theorem} \ $($\cite[Theorem 2.2]{Hovey02}$)$ \ Let $\A$ be a bicomplete abelian category. Then there is a bijection:
 \[
\left\{\begin{gathered}\text{abelian model structures  }\\ \text{on  $\A$}
\end{gathered}\;
\right\} \xymatrix@C=1pc{ \ar[r]^-{\sim} & }\left\{
\begin{gathered}
  \text{Hovey triples of $\A$}
\end{gathered}
\right\}\,
\]
which sends an abelian model structure of $\A$ to $(\A_c, \A_{triv}, \A_f)$.
\end{theorem}

\vskip10pt
Now let $\A$ be an abelian model category. By Hovey's correspondence, there are two complete cotorsion pairs $(\A_c, \A_{tri}\cap \A_f)$ and $(\A_c\cap \A_{triv}, \A_f)$. Note that both $\A_f$ and $\A_c$ are additive subcategories of $\A$. Let $\omega=\A_{cf}\cap \A_{triv}$. Then $\omega$ is contravariantly finite in $\A_f$. In fact, for each $A\in \A_f$, there is a short exact sequence $0\to F \to C\stackrel{f}\to A\to 0$ such that $C\in \A_c\cap \A_{triv}$ and $F\in \A_f$ by the completeness of the cotorsion pair $(\A_c\cap \A_{triv}, \A_f)$. Since $\A_f$ is closed under extension, we know that $C\in \omega$. Moreover, for any $W\in \A_{tri}$, applying $\Hom_\A(W,-)$ to the above exact sequence we obtain a long exact sequence $$0\to Hom_\A(W,F)\to \Hom_\A(W, C)\stackrel{f_*}\to \Hom_\A(W, A)\to {\rm Ext}^1_\A(W, F)=0. $$
That is to say $f_*$ is a surjection and thus $f$ is a right $\omega$-approximation of $A$. Dually $\omega$ is covariantly finite in $\A_c$.
Similarly, we know that $\A_{cf}$ is also contravariantly finite in $\A_f$ and covariantly finite in $\A_c$.

In general, $\A_f$ and $\A_c$ need not to be closed under taking kernels of $\omega$-epics and taking cokernels of $\omega$-monics, respectively. But $\A_f$ is closed under taking kernels of special $\omega$-epics in $\A_f$ and $\A_c$ is closed under taking cokernels of special $\omega$-monics.

\begin{proposition} \ Let $\A$ be an abelian model category and $(\A_c, \A_{triv}, \A_f)$ the corresponding Hovey triple. Let $\omega=\A_{cf}\cap \A_{triv}$ Then
\vskip5pt
$(i)$ \ $\omega$ is contravariantly finite in $\A_f$ and any special $\omega$-epic in $\A_f$ has a kernel in $\A_f$. Thus the stable category $\A_f/\omega$ has a left triangulated structure induced by $\omega$.
\vskip5pt
$(ii)$ \ $\A_{cf}$ is contravariantly finite in $\A_f$ and any special $\A_{cf}$-epic in $\A_f$ has a kernel in $\A_f$. Thus the stable category $\A_f/\A_{cf}$ has a left triangulated structure induced by $\A_{cf}$.
\vskip5pt
$(iii)$ \ $\omega$ is covariantly finite in $\A_c$ and any special $\omega$-monic in $\A_c$ has a cokernel in $\A_c$. Thus the stable category $\A_c/\omega$ has a right triangulated structure induced by $\omega$.
\vskip5pt
$(iv)$ \ $\A_{cf}$ is covariantly finite in $\A_c$ and any special $\A_{cf}$-monic in $\A_c$ has a cokernel in $\A_c$. Thus the stable category $\A_c/\A_{cf}$ has a right triangulated structure induced by $\A_{cf}$.
\end{proposition}

\begin{proof} \ We only prove the statement $(i)$ and $(ii)$.  The others can be obtained from the first by duality.

$(i)$ \ By Theorem 3.6, we only need to prove that any special $\omega$-epic in $\A_f$ has a kernel in $\A_f$. By above discussion, for each $A\in \A_f$ we can assign a right $\omega$-approximation $p_A: W_A\to A$ such that $p_{_A}$ is a fibration, i.e. $p_{_A}$ is an epimorphism with $K_A=\ker p_{_A}\in \A_f$.  For any special $\omega$-epic morphism $(f, p_{_A}): B\oplus A\to A $, there is pullback square
\[
\xymatrix{
C \ar[r]^g \ar[d]& B \ar[d]^f \\
W_A \ar[r]_{p_{_A}} & A. }
\]
Then $C$ is the kernel of $(f, p)$. Since the class of fibrations is stable under base change \cite[Proposition 3.14]{Dwyer/Spalinski95} we know that $g$ is also a fibration, then there is a short exact sequence
$$0\to \ker g\to C\stackrel{g}\to B\to 0$$  with $\ker g\in \A_f$. Since $\A_f$ is closed under extensions, we obtain that $C\in \A_f$, i.e. the special $\omega$-epic morphism $(f, p_{_A})$ has a kernel in $\A_f$.

The proof of $(ii)$ is similarly.
\end{proof}

\begin{example} \  Let $\A=R$-Mod. where $R$ is a Gorenstein ring. Then there exists an abelian model structure on $R$-Mod, called {\it Gorenstein projective model structure} with $\A_c={\rm Gproj}(R)$ (the class of Gorenstein projective $R$-modules), $\A_{triv}=\mathcal{P}^{<\infty}(R)$ (the modules of finite projective dimension) and $\A_f=\A$, \cite[Theorem 8.6]{Hovey02}. In this case, the class of trivial bifibrant objects is the class of projective $R$-modules.The four stable categories are  $\A_f/\omega=R\mbox{-}\ul{{\rm Mod}}$, $\A_f/\A_{cf}=R\mbox{-}{\rm Mod}/{\rm Gproj}(R)$, $\A_c/\omega=\ul{{\rm Gproj}}(R), 0$. Dually, there exists an abelian model structure on $R$-Mod, called {\it Gorenstein injective model structure} with $\A_f/\omega={\rm Ginj}(R)$ (the class of Gorenstein injective $R$-modules), $\A_{triv}=\mathcal{I}^{<\infty}(R)$ (the modules of finite injective dimension) and $\A_c=\A$, \cite[Theorem 8.6]{Hovey02}. The four stable categories are  $\A_c/\omega=R\mbox{-}\overline{{\rm Mod}}$, $\A_c/\A_{cf}=R\mbox{-}{\rm Mod}/{\rm Ginj}(R)$, $\overline{{\rm Ginj}}(R), 0$.
We want to point out that the left or right triangulated structure of these stable categories can also be obtained by Beligiannis-Marmaridis's Theorem \cite[Theorem 2.12]{Beligiannis/Marmaridis94}.
\end{example}
\vskip10pt
\subsection{The pretriangulated structures of the homotopy categories of abelian model categories}
The homotopy category $\Ho(\A)$ of a model category $\A$ is the localization of $\A$ with respect to the class of weak equivalences \cite{Gabriel/Zisman67}.
By Quillen's theorem \cite[Theorem I.1]{Quillen67}, the homotopy category $\Ho(\A)$ is also equivalent to both the homotopy category $\Ho(\A_f)$ and $\Ho(\A_c)$. We recall the construction of the homotopy categories concretely. For each $A\in \A$, applying factorization axiom, we can factor the morphism $0\to A$ as the composite $0\to Q(A)\stackrel{q_{_A}}\to A$ with $Q(A)\in \A_c$ and $q_{_A}\in \mathcal{W}e\cap \mathcal{F}ib$. If $A$ is already a cofibant object, we take $Q(A)=A$. By construction, $Q(A)$ is in $ \A_{cf}$ if $A\in \A_f$. $Q$ becomes a functor from $\A_f$ to the quotient category $\pi\A_{cf}=\A_{cf}/\stackrel{h}\sim$ \cite[Lemma 5.1, Remark 5.2]{Dwyer/Spalinski95}, where $\stackrel{h}\sim$ denotes the homotopy relation which is defined via cylinder or path objects \cite[Section 4]{Dwyer/Spalinski95}, see also Section 5.1 below. Dually, there is a functor $R:\A_c\to \pi\A_{cf}$ which is constructed by the factorization of $A\to 0$ as an acyclic cofibration followed by a fibration $A\stackrel{i^A}\to R(A)\to 0$ for any $A\in \A$. The localization $\gamma: \A\to \Ho(\A)$ is constructed by sending each object $A$ to $A$ and each morphism $f: A\to B$ to $\gamma(f)=[RQ(f)]$ the homotopy class of $RQ(A)$ in $\pi\A_{cf}$, and as category, $\Ho(\A)$ is equivalent to the quotient category $\pi\A_{cf}$
\cite[Theorem I.1.1]{Quillen67}, \cite[Remark 5.7, Theorem 6.2]{Dwyer/Spalinski95}. In the homotopy category $\Ho(\A)$, any morphism $f:A\to B$ is of the form $\gamma(q_{_B})\gamma(i^{Q(B)})^{-1}\gamma(f')\gamma(i^{Q(A)})\gamma(q_{_A})^{-1}$ with $f'\in \Hom_{\A}(RQ(A), RQ(B))$ \cite[Proposition 5.8]{Dwyer/Spalinski95}. By the constructions of $Q$ and $R$ we know that each morphism $f:A\to B$ in $\Ho(\A_f)$ (respectively $\Ho(\A_c)$) can be written as $\gamma(\alpha)\gamma^{-1}(p_{_A})$ (respectively $\gamma(i^{B})^{-1}\gamma(\alpha)$).

If $\A$ is an abelian model category, with the notations as in Section 4.2, the homotopy category $\Ho(\A)$ is equivalent to the stable category $\A_{cf}/\omega$ \cite[Theorem VIII,4.2]{Beligiannis/Reiten07}, \cite[Proposition 4.3,4.7]{Gillespie11}, \cite[Proposition 1.1.14]{Becker12}.The full subcategory $\A_f$ becomes an {\it  additive fibration category } and $\A_c$ becomes an {\it additive cofibration category} in the sense of Definition 5.1, Remark 5.2. The homotopy category $\Ho(\A_f)$ is a left triangulated category \cite[Section I.4, Theorem 3]{Brown74}. The homotopy category $\Ho(\A_c)$ is a right triangulated category by duality. Next we will see that the left (respectively right) triangulated structure of $\A_f/\omega$ ($\A_c/\omega$) induces a left triangulated structure on $\Ho(\A_f)$ ( respectively $\Ho(A_c)$).


Recall that, for each $A\in \A_f$, we can assign a right $\omega$-approximation $p_{_A}:W_A\to A$ with $\ker p_{_A}=K_A\in \A_f$. The loop functor of $\A_f/\omega$ is induced by the functor $\Omega: \A_f\to \A_f/\omega$ which sends each $A\in \A_f$ to $\ul{K}_A$ and each morphism $f:A\to B$ to $\ul{\kappa}_f$, see Section 2.2.  Dually, for each $A\in \A_c$, we can assign a left $\omega$-approximation $A\stackrel{i^A}\to W^A$ with cokernel $K^A\in \A_c$. The suspension functor of $\A_c/\omega$ is induced by the functor $\Sigma: \A_c\to \A_c/\omega$ as defined in Section 2.2.

\begin{lemma} \ Let $\A$ be an abelian model category. With above notations, we have :
\vskip5pt
$(i)$ \ If $\ul{f} =\ul{g}$ in $\A_f/\omega$, then $f$ is right homotopic to $g$ in $\A_f$ and thus $\gamma(f)=\gamma(g)$ in $\Ho(\A_f)$.
\vskip5pt
$(ii)$ \ If $\ul{f} =\ul{g}$ in $\A_c/\omega$, then $f$ is left homotopic to $g$ in $\A_c$ and thus $\gamma(f)=\gamma(g)$ in $\Ho(\A_c)$.
\vskip5pt
$(iii)$ \ Let $f: A\to B$ be a morphism in $\A_f$, then $\gamma(f)=\gamma(g)$ in $\Ho(\A_f)$ if and only if $\ul{fq_{_A}}=\ul{gq_{_A}}$ in $\A_{f}/\omega$
\vskip5pt
$(iv)$ \ Let $f: A\to B$ be a morphism in $\A_c$, then $\gamma(f)=\gamma(g)$ in $\Ho(\A_c)$ if and only if $\ul{i^{B}f}=\ul{i^{B}g}$ in $\A_{c}/\omega$
\end{lemma}

\begin{proof} \ We only prove $(i)$ and $(iii)$. The other two statements are dual.
 $(i)$ \  For the definition of right homotopic relation see Section 5.1. Assume that $\ul{f}=\ul{g}: A\to B$ in $\A_f/\omega$, then there is a factorization of $g-f$: $A\stackrel{t}\to W\stackrel{s}\to B$ with $W\in \omega$. Note that $B\oplus W_B$ is a path object of $B$ since we have factorization  $B\stackrel{\left(\begin{smallmatrix}
1 \\
0
\end{smallmatrix}\right)}\to B\oplus W_B\stackrel{\left(\begin{smallmatrix}
1 & p_{_B}\\
1 & 0
\end{smallmatrix}\right)}\to B\oplus B$. Since $p_{_B}$ is a right $\omega$-approximation of $B$, there exists a morphism $l: W\to W_B$ such that $p_{_B}l=s$.
Then we have a commutative diagram
\[
\xymatrix{
A\ar[d]_{\left(\begin{smallmatrix}
f \\
lt
\end{smallmatrix}\right)} \ar[r]^{\left(\begin{smallmatrix}
g \\
f
\end{smallmatrix}\right)} & B\oplus B  \\
B\oplus W_B \ar[ru]_{\left(\begin{smallmatrix}
1 & p_{_B}\\
1 & 0
\end{smallmatrix}\right)}}
\]
so $\left(\begin{smallmatrix}
f \\
lt
\end{smallmatrix}\right)$ is a right homotopy from $f$ to $g$.

$(iii)$. \  By construction, $\gamma(f)=\gamma(g)$ if and only if $Q(f)$ is homotopic to $Q(g)$ in $\Ho(\A)$ and thus if and only if $Q(f)$ is stable equivalent to $Q(g)$ in $\A_{cf}/\omega$. Since $(f-g)q_{_A}=q_{_B}(Q(f)-Q(g))$ \cite[Lemma 5.1]{Dwyer/Spalinski95}, thus if $\ul{Q(f)}=\ul{Q(g)}$ in $\A_{cf}/\omega$, then $Q(f)-Q(g)$ factors through $\omega$ and then so is $fq_{_A}-gq_{_A}$, i.e. $\ul{fq_{_A}}=\ul{gq_{_A}}$ in $\A_{f}/\omega$. Conversely, if  $\ul{fq_{_A}}=\ul{gq_{_A}}$ in $\A_{f}/\omega$, by $(i)$,  $\gamma(fq_{_A})=\gamma(gq_{_A})$ in $\Ho(\A_f)$ and then $\gamma(f)=\gamma(g)$ since $\gamma(q_{_A})$ is an isomorpshim.
\end{proof}

\begin{lemma} \ $(i)$ \ The functor $\Omega: \A_f\to \A_f/\omega$ induces a functor $\A_f\to \Ho(\A_f)$ which sends weak equivalences into isomorphisms.

$(ii)$ \ The functor $\Sigma: \A_c\to \A_c/\omega$ induces a functor $\A_c \to \Ho(\A_c)$ which sends weak equivalences into isomorphisms.
\end{lemma}

\begin{proof} \ We only prove $(i)$ since the assertion $(ii)$ is dual. Denote by $\gamma: \A_f\to \Ho(\A_f)$ the localization functor as constructed in Section 5.1.
Define a functor $F:\A_f\to \Ho(\A_f)$ which sends each object $A\in \A_f$ to $\Omega(A)$ and each morphism $f:A\to B$ to $\gamma(\kappa_f)$. This functor is well defined. In fact, recall the construction of $\kappa_f$:
\[
\xymatrix{
0\ar[r] & \Omega(A)\ar[r]^{\iota_A}\ar[d]_{\kappa_f} &W_A\ar[d]_{x_f} \ar[r]^{p_{_{A}}} & A\ar[d]_f \ar[r] &0 \\
0 \ar[r] & \Omega(B)\ar[r]^{\iota_B} & W_B \ar[r]^{p_{_{B}}} & B\ar[r]& 0.}
\]
where $p_{_A}, p_{_B}$ are the assigned right $\omega$-approximations of $A$ and $B$, respectively. If there is another pair of morphisms $x'_f$ and then $k'_f$ satisfy the above commutative diagram, by the discussion of Section 2.2, $\kappa_f-\kappa'_f$ factors through $W_A\in \omega$. Then $\kappa_f$ is right homotopic to $\kappa'_f$ by $(i)$ of Lemma 4.9. Thus $\gamma(\kappa_f)=\gamma(\kappa'_f)$.

Next we will show that $F$ takes weak equivalences into isomorphisms. It is enough to prove that if $f$ is a weak equivalence then so is $\kappa_f$.  Note that $\left(\begin{smallmatrix}
1 \\
0
\end{smallmatrix}\right):A\to A\oplus W_A$ and $\left(\begin{smallmatrix}
1 \\
0
\end{smallmatrix}\right): B\to B\oplus W_B$ are acyclic cofibrations, thus $\left(\begin{smallmatrix}
f & 0 \\
0 & x_f
\end{smallmatrix}\right)$ is also a weak equivalence by the two out of three axiom since $\left(\begin{smallmatrix}
f & 0 \\
0 & x_f
\end{smallmatrix}\right)\left(\begin{smallmatrix}
1 \\
0
\end{smallmatrix}\right)=\left(\begin{smallmatrix}
1 \\
0
\end{smallmatrix}\right)f$ is a weak equivalence.

Observe that $\left(\begin{smallmatrix}
1 & p_{_A} \\
1 & 0
\end{smallmatrix}\right): A\oplus W_A\to A\oplus A$ is a fibration with kernel $K_A\stackrel{\left(\begin{smallmatrix}
0 \\
\iota_A
\end{smallmatrix}\right)}\to A\oplus W_A$ and $\left(\begin{smallmatrix}
1 & p_{_B} \\
1 & 0
\end{smallmatrix}\right): B\oplus W_B\to A\oplus A$ is a fibration  with kernel $K_B\stackrel{\left(\begin{smallmatrix}
0 \\
\iota_B
\end{smallmatrix}\right)}\to B\oplus W_B$.
Consider the commutative diagram
\[
\xymatrix{
 A\oplus W_A\ar[r]^{\left(\begin{smallmatrix}
1 & p_{_A} \\
1 & 0
\end{smallmatrix}\right)}\ar[d]_{\left(\begin{smallmatrix}
f & 0 \\
0 & x_f
\end{smallmatrix}\right)} &A\oplus A\ar[d]_{\left(\begin{smallmatrix}
f & 0 \\
0 & f
\end{smallmatrix}\right)} & 0 \ar[l] \ar@{=}[d]  \\
 B\oplus W_B\ar[r]_{\left(\begin{smallmatrix}
1 & p_{_B} \\
1 & 0
\end{smallmatrix}\right)} & B\oplus B & 0\ar[l]}
\]
all the columns are weak equivalences, thus by the Gluing Lemma \cite[Lemma 1.4.1 (2)]{Radulescu-Banu06}, we know that $\kappa_f: K_A\to K_B$ is also a weak equivalence.
\end{proof}

\begin{remark} \ By this Lemma we know that the functor $F$ constructed as in the proof induces a functor which we still denote it by $\Omega: \Ho(\A_f)\to \Ho(\A_f)$ by the universal property of $\gamma$. Dually, the induced functor as in $(ii)$, we still denote it as $\Sigma: \Ho(\A_c)\to \Ho(\A_c)$.
\end{remark}

We define the distinguished left triangles in $\Ho(\A_f)$ as those which are isomorphic to the forms of $\Omega(A)\stackrel{\gamma(\zeta_f)}\to {\rm PB}(f)\stackrel{\gamma(\eta_f)}\to B\stackrel{\gamma(f)}\to A $, where $\Omega(A)\stackrel{\ul{\zeta}_f}\to \ul{\rm PB}(f)\stackrel{\ul{\eta}_f}\to \ul{B}\stackrel{\ul{f}}\to \ul{A} $ is a distinguished left triangle in $\A_f/\omega$.   Let $\triangle_l$ be the class of distinguished left triangles. Dually, we define the class $\triangle^r$ of distinguished right triangles of $\A_c/\omega$.

In the rest of this section, for simplicity, we will omit the underline on the object in stable categories but preserve the underline on the morphisms.

\begin{theorem} \ $(i)$ The pair $(\Omega, \triangle_l)$ is a left triangulated structure on $\Ho(\A_f)$.

\vskip5pt
$(ii)$ The pair $(\Sigma, \triangle^r)$ is a right triangulated structure on $\Ho(\A_c)$.
\end{theorem}
\begin{proof} \ $(i)$ \ We prove the axioms $({\rm LT}1)-({\rm LT}4)$ of Definition 3.1 one by one.  Recall that any morphism in $\Ho(\A_f)$ starting from $A$ is of the form $\gamma(\alpha)\gamma(q_{_A})^{-1}$.

$({\rm LT}1)$ \ For any morphism $f: A\to B$ in $\Ho(\A_f)$, $f$ is of the form $\gamma(\alpha)\gamma(q_{_A})^{-1}$. For $\alpha: Q(A)\to B$ there exists a distinguished left triangle $\Omega(B)\stackrel{\ul{\zeta}_\alpha}\to \rm PB(\alpha)\stackrel{\ul{\eta}_\alpha}\to Q(A)\stackrel{\ul{\alpha}}\to B $ in $\A_f/\omega$. We have the following commutative diagram
\[
\xymatrix{
\Omega(B)\ar[r]^{\gamma(\zeta_\alpha)}\ar@{=}[d] & {\rm PB}(\alpha)\ar[r]^{\gamma(\eta_\alpha)}\ar@{=}[d] & Q(A)\ar[d]_{\gamma(q_{_A})} \ar[r]^{\gamma(\alpha)} & B\ar@{=}[d]  \\
\Omega(B) \ar[r]^{\gamma(\zeta_{\alpha})} & {\rm PB}(\alpha)\ar[r]^{\gamma(q_{_A}\eta_\alpha)} & A \ar[r]^{f} & B}
\]
in $\Ho(\A_f)$. Since $\gamma(q_{_A})$ is an isomorphism, thus $\Omega(B) \stackrel{\gamma(\zeta_{\alpha})} \To {\rm PB}(\alpha)\stackrel{\gamma(q_{_A}\eta_\alpha)}\To A \stackrel{f} \to B \in \triangle_l$ by definition. We should remind the reader that the distinguished left triangle of $f$ is independently of the choice of $\alpha$ by $(iii)$ of Lemma 4.9.

 For any $A\in \Ho(\A_f)$, $0\to A\stackrel{1_A}\to A\to 0\in \triangle_l$ since $0\to A\stackrel{\ul{1}_A}\to A\to 0$ is a distinguished left triangle in $\A_f/\omega$.

$({\rm LT}2)$ \ Without loss of generality, it is enough to prove that the left triangle $\Omega(A)\stackrel{-\Omega(\gamma(f))}\To \Omega(B)\stackrel{\gamma(\zeta_f)}\To {\rm PB}(f)\stackrel{\gamma(\eta_f)}\To B\in \triangle_l$ for a distinguished left triangle $\Omega(B)\stackrel{\gamma(\zeta_f)}\To {\rm PB}(f)\stackrel{\gamma(\eta_f)}\To A\stackrel{\gamma(f)}\To B $. Since $\Omega(B)\stackrel{-\ul{\Omega(f)}}\To \Omega(A)\stackrel{\ul{\zeta}_f}\To {\rm PB}(f)\stackrel{\ul{\eta}_f}\To \ul{A} $ is a distinguished left triangle in $\A_f/\omega$, we have $\Omega(A)\stackrel{-\gamma(\Omega(f))}\To \Omega(B)\stackrel{\gamma(\zeta_f)}\To {\rm PB}(f)\stackrel{\gamma(\eta_f)}\To B\in \triangle_l$. By the construction of $\Omega$ on $\Ho(\A_f)$, $\Omega(\gamma(f))=\gamma(\Omega(f))$, then $\Omega(A)\stackrel{-\Omega(\gamma(f))}\To \Omega(B)\stackrel{\gamma(\zeta_f)}\To {\rm PB}(f)\stackrel{\gamma(\eta_f)}\To B\in \triangle_l$.

$({\rm LT}3)$ \ Without loss of generality, we only consider the following  diagram of distinguished left triangles in $\Ho(\A_f)$ satisfying $h\gamma(f)=\gamma(f')g$:
\[
\xymatrix{
\Omega(B)\ar[r]^{\gamma(\zeta_f)}\ar[d]^{\Omega(h)} & {\rm PB}(f)\ar[r]^{\ \  \gamma(\eta_f)} & A \ar[d]^{g} \ar[r]^{\gamma(f)} & B\ar[d]^h  \\
\Omega(B') \ar[r]^{\gamma(\zeta_{f'})} & {\rm PB}(f')\ar[r]^{\ \ \gamma(\eta_{f'})} & A' \ar[r]^{\gamma(f')} & B'.}
\]
We start with the special case where $g=\gamma(\alpha)$ and $h=\gamma(\beta)$. Since $\gamma(\beta f)=\gamma(f'\alpha)$, by $(iii)$ of Lemma 4.9, we have $\ul{\beta fq_{_A}}=\ul{f'\alpha q_{_A}}$ in $\A_f/\omega$. So there is a commutative diagram of distinguished left triangles in $\A_f/\omega:$
\[
\xymatrix{
\Omega(B)\ar[r]^{\ul{\zeta}_{fq_{_A}}}\ar[d]^{\ul{\Omega(\beta)}} & {\rm PB}(fq_{_A})\ar[r]^{\ \  \ul{\eta}_{fq_{_A}}} \ar[d]^{\ul{u}}& Q(A) \ar[d]_{\ul{\alpha q_{_A}}} \ar[r]^{\ul{fq_{_A}}} & \ul{B}\ar[d]^{\ul{\beta}}  \\
\Omega(B') \ar[r]^{\ul{\zeta}_{f'}} & {\rm PB}(f')\ar[r]^{\ \ \ul{\eta}_{f'}} & A' \ar[r]^{\ul{f'}} & B'.}
\]
Consider the commutative diagram in $\A_f$
\[
\xymatrix{
 W_B\ar[r]^{p_{_B}}\ar@{=}[d] & B \ar@{=}[d]& Q(A) \ar[l]_{fq_{_A}} \ar[d]^{q_{_A}}  \\
 W_B\ar[r]_{p_{_B}} & B & A\ar[l]^{f}}
\]
Since $p_{_A}$ is a fibration and $q_{_A}$ is a weak equivalence, by the Gluing Lemma \cite[Lemma 1.4.1 (2)]{Radulescu-Banu06}, the morphism $v: {\rm PB}(fq_{_A})\to {\rm PB}(f) $ is also a weak equivalence and thus induces the following commutative diagram of isomorphic distinguished left triangles
\[
\xymatrix{
\Omega(B)\ar[r]^{\ul{\zeta}_{fq_{_A}}}\ar@{=}[d] & {\rm PB}(fq_{_A})\ar[r]^{\ \  \ul{\eta}_{fq_{_A}}} \ar[d]^{\ul{v}}& Q(A) \ar[d]_{\ul{ q_{_A}}} \ar[r]^{\ul{fq_{_A}}} & B\ar@{=}[d]  \\
\Omega(B) \ar[r]^{\ul{\zeta}_{f}} & {\rm PB}(f)\ar[r]^{\ \ \ul{\eta}_{f}} & A \ar[r]^{\ul{f}} & B.}
\]
Then $\gamma(u)\gamma(v)^{-1}: {\rm PB}(f)\to {\rm PB}(f')$ is the desired filler. For the general case, let $h=\gamma(\alpha)\gamma(q_{_B})^{-1}$. Take a pullback diagram in $\A_f$:
\[
\xymatrix{
C \ar[r]^{s} \ar[d]_l& A \ar[d]^f \\
Q(B) \ar[r]_{\ \ \ q_{_B}} & B. }
\]
Since $q_{_B}$ is an acyclic fibration, so is $s$. Write $g\gamma(s)$ as $\gamma(\beta)\gamma(q_{_C})^{-1}$, then in $\Ho(\A_f)$, we have $\gamma(f') \gamma(\beta)=\gamma(f')g\gamma(s)\gamma(q_{_C})=h\gamma(f)\gamma(s)\gamma(q_{_C})=\gamma(\alpha)\gamma(q_{_B})^{-1} \gamma(f)\gamma(s)\gamma(q_{_C})=\gamma(\alpha)\gamma(lq_{_C})$. Thus by the special case, there is a morphism $m:{\rm PB}(lq_{_C})\to {\rm PB}(f')$ in $\Ho(\A_f)$ such that the following diagram of distinguished left triangles commutes
\[
\xymatrix{
\Omega(Q(B))\ar[r]^{\gamma(\zeta_{_{lq_{_C}}})}\ar[d]^{\Omega(\gamma(\alpha))} & {\rm PB}(lq_{_C})\ar[r]^{\ \  \ \gamma(\eta_{_{lq_{_C}}})}\ar[d]^{m} & Q(C) \ar[d]^{\gamma(\beta)} \ar[r]^{\gamma(lq_{_C})} & B\ar[d]^{\gamma(\alpha)}  \\
\Omega(B') \ar[r]^{\gamma(\zeta_{f'})} & {\rm PB}(f')\ar[r]^{\ \ \gamma(\eta_{f'})} & A' \ar[r]^{\gamma(f')} & B'.}
\]
 Similarly to the construction of $v$ in the proof of the special case, there exists weak equivalence $n:{\rm PB}(lq_{_C})\to {\rm PB}(f)$  induced by the commutative diagram
 \[
\xymatrix{
A \ar[r]^{f}& B \\
Q(C)\ar[u]^{sq_{_C}} \ar[r]^{ lq_{_C}} & Q(B)\ar[u]_{q_{_B}} }
\]
in $\A_f$.  The composite $m \gamma(n)^{-1}:{\rm PB}(f)\to {\rm PB}(f')$ is the desired filler.

 $({\rm LT}4)$ \ The octahedral axiom. \ Given two composable morphisms $C\stackrel{g}\to B\stackrel{f}\to A$ in $\Ho(\A_f)$. If $g=\gamma(\beta)$ and $f=\gamma(\alpha)$, applying the octahedral axiom in $\A_f/\omega$ to the morphisms $C\stackrel{\beta}\to B\stackrel{\alpha}\to A$, we know that in this special case, the octahedral axiom holds in $\Ho(\A_f)$. In the general case, we write $f=\gamma(\alpha)\gamma(q_{_B})^{-1}$ and $\gamma(q_{_B})^{-1}g=\gamma(\beta)\gamma(q_{_C})^{-1}$. By the special case, there is a commutative diagram of distinguished left triangles:
 \[
\xymatrix{
& \Omega({\rm PB}(\alpha))\ar[d]^{\gamma(\zeta_{_\beta})\Omega(\gamma(\eta_\alpha))} &  & \\
\Omega(Q(B))\ar[r]^{\gamma(\zeta_{_\beta})}\ar[d]_{\Omega(\gamma(\alpha))} & {\rm PB}(\beta)\ar[r]^{\ \ \gamma(\eta_{_\beta})}\ar[d]^{n} & Q(C)\ar@{=}[d] \ar[r]^{\gamma(\beta)} & Q(B)\ar[d]^{\gamma(\alpha)}  \\
\Omega(A) \ar[r]^{\gamma(\zeta_{_\alpha\beta})\ }\ar@{=}[d] &{\rm PB}(\alpha\beta)\ar[r]^{\ \ \gamma(\eta_{_\alpha\beta})}\ar[d]^m & Q(C) \ar[r]^{\gamma(\alpha\beta)} \ar[d]^{g}& A\ar@{=}[d]\\
\Omega(A)\ar[r]^{\gamma(\zeta_\alpha)} & {\rm PB}(\alpha)\ar[r]^{\gamma(\eta_\alpha)}& Q(B)\ar[r]^{\gamma(\alpha)}& A.}
\]
By $({\rm LT}1)$, the distinguished left triangle of $f$ is isomorphic to the one of $\gamma(\alpha)$:
\[
\xymatrix{
\Omega(B)\ar[r]^{\gamma(\zeta_\alpha)}\ar@{=}[d] & {\rm PB}(\alpha)\ar[r]^{\gamma(\eta_\alpha)}\ar@{=}[d] & Q(A)\ar[d]_{\gamma(q_{_A})} \ar[r]^{\gamma(\alpha)} & B\ar@{=}[d]  \\
\Omega(B) \ar[r]^{\gamma(\zeta_{\alpha})} & {\rm PB}(\alpha)\ar[r]^{\gamma(q_{_A}\eta_\alpha)} & A \ar[r]^{f} & B.}
\]
Similarly, since $g=\gamma(q_{_B}\beta)\gamma(q_{_C})^{-1}$, we know that the distinguished left triangle of $g$ is isomorphic to  $\Omega(B)\stackrel{\gamma(\zeta_{_{q_{_B}\beta}})}\To {\rm PB}(q_{_B}\beta)\stackrel{\gamma(\eta_{_{q_{_B}\beta}})}\To Q(C)\stackrel{\gamma(q_{_B}\beta)}\to B $. Applying the Gluing Lemma \cite[Lemma 1.4.1 (2)]{Radulescu-Banu06} to the following commutative diagram
\[
\xymatrix{
 Q(B)\oplus W_{Q(B)}\ar[r]^{\left(\begin{smallmatrix}
1 & p_{_{Q(B)}} \\
1 & 0
\end{smallmatrix}\right)}\ar[d]_{\left(\begin{smallmatrix}
q_{_B} & 0 \\
0 & x_{_{q_{_B}}}
\end{smallmatrix}\right)} &Q(B)\oplus Q(B)\ar[d]_{\left(\begin{smallmatrix}
q_{_B} & 0 \\
0 & q_{_B}
\end{smallmatrix}\right)} & Q(C) \ar[l]_{\ \ \ \ \ \ \left(\begin{smallmatrix}
\beta  \\
0
\end{smallmatrix}\right)} \ar@{=}[d]  \\
 B\oplus W_B\ar[r]_{\left(\begin{smallmatrix}
1 & p_{_B} \\
1 & 0
\end{smallmatrix}\right)} & B\oplus B & Q(C)\ar[l]^{\ \ \ \ \ \left(\begin{smallmatrix}
q_{_B}\beta  \\
0
\end{smallmatrix}\right)}}
\]
(for $\left(\begin{smallmatrix}
q_{_B} & 0 \\
0 & x_{_{q_{_B}}}
\end{smallmatrix}\right)$ is a weak equivalence, see the proof of Lemma 4.10) we know that there is a weak equivalence ${\rm PB}(\beta)\to {\rm PB}(q_{_B}\beta)$ which gives an isomorphism from the distinguished left triangle of $\gamma(\beta)$ to the above one.

Since we can write $fg=\gamma(\alpha\beta)\gamma(q_{_C})^{-1}$, we know that the distinguished left triangle of $fg$ is isomorphic to the one of $\gamma(\alpha\beta)$. So the octahedral axiom holds for the morphisms $f$ and $g$ in $\Ho(\A_f)$. This completes the proof of the assertion $(i)$.

The statement of $(ii)$ can be obtained by duality.
\end{proof}

\begin{lemma}  In the situation of Proposition 4.12, we have
\vskip5pt
$(i)$ \ $(\Sigma Q, \Omega R)$ is an adjoint pair on $\Ho(\A)$;
\vskip5pt
$(ii)$ \ $(R\Sigma , Q\Omega )$ is an adjoint pair on $\A_{cf}/\omega$.

\end{lemma}
\begin{proof} \ $(i)$ \ Firstly, we prove that there is a natural isomorphism
$$\Hom_{\Ho(\A)}(\Sigma(A), B)\cong \Hom_{\Ho(\A)}(A, \Omega(B))$$
for $A\in \A_c, B\in \A_f$. In this case, any morphism from $\Sigma(A)$ to $B$ is of the form $\gamma(\alpha)$ for some $\alpha: \Sigma(A)\to B$ in $\A$, see the beginning of  Section 4.3 or \cite[Corollary I.1.1]{Quillen67}. For each $\gamma(\alpha)\in \Hom_{\Ho(\A)}(\Sigma(A),B)$, by the construction of $\Sigma(A)$ and $\Omega(B)$, there is a commutative diagram
\[
\xymatrix{
0\ar[r] & A\ar[r]\ar[d]_{\kappa_\alpha} & W^A\ar[d] \ar[r] & \Sigma(A)\ar[d]_\alpha \ar[r]& 0   \\
0 \ar[r] & \Omega(B)\ar[r]^{\iota_B} & W_B \ar[r]^{p_{_{B}}} & B \ar[r] &0. }
\]
The existence of the middle morphism is since $p_{_B}$ is a right $\omega$-approximation and $W^{A}\in \omega$. By the discussion in Section 2.2, we know that $\ul{\kappa}_\alpha$ is uniquely decided by the stable equivalence classes of $\alpha$ in $\A/\omega$. Since $A, \Sigma(A)\in \A_c$ and $B, \Omega(B)\in \A_f$, we know that the stable equivalence relations defined via $\omega$ is just the homotopy equivalence relation \cite[Proposition 4.4]{Gillespie11}. So $\gamma(\kappa_\alpha)$ is uniquely decided by $\gamma(\alpha)$. Dually, for each $\gamma(\beta)\in \Hom_{\Ho(\A)}(A,\Omega(B))$, there is a uniquely morphism $\gamma(\kappa^\beta): \Sigma(A)\to B $.
So if we define the map
$$\varphi: \Hom_{\Ho(\A)}(\Sigma(A), B)\to \Hom_{\Ho(\A)}(A, \Omega(B))$$
which sends $\gamma(\alpha)$ to $-\gamma(\kappa_{\alpha})$, then it has an inverse which sends $\gamma(\beta)$ to $-\gamma(\kappa^{\beta})$. The naturalness of $\varphi$ can be verified directly by the description of morphisms in $\Ho(\A_f)$ and $\Ho(\A_c)$.

In general, for $A, B\in \A$, we have natural isomorphisms
\begin{align*}
\Hom_{\Ho(\A)}(\Sigma Q(A), B)& \cong \Hom_{\Ho(\A)}(\Sigma Q(A),R(B)) \\
 & \cong \Hom_{\Ho(\A)}(Q(A), \Omega R(B)) \\
 & \cong  \Hom_{\Ho(\A)}(A, \Omega R(B))
\end{align*}
where the first and the last isomorphisms are given by the isomorphisms $\Sigma(A)\to R(\Sigma(A))$ and $Q(\Omega(B))\to \Omega(B)$ respectively, and the middle one is by the above proof.

\vskip5pt
$(ii)$ \ Similarly to the proof of $(i)$ we have adjunction isomorpshim
$$\varphi: \Hom_{\A_{cf}/\omega}(R\Sigma(A), B)\to \Hom_{\A_{cf}/\omega}(A, Q\Omega(B))$$
which sends $\ul{\alpha}$ to $-\ul{Q(\kappa_{\alpha i^{\Sigma A}})}$ where $i^{\Sigma A}: \Sigma A\to R\Sigma A$ is the acyclic cofibration and $\kappa_{\alpha i^{\Sigma A}}$ is constructed as in the proof of $(i)$. The inverse of $\varphi$ sends a morphism $
\ul{\beta}$ to $-R(\kappa^\beta)$.
\end{proof}
\vskip10pt
Recall that a {\it pretriangulated structure} on an additive category $\mcC$ is a quadruple $(\Omega, \Sigma, \triangle_l, \triangle^r)$ such that
$(i)$ \ $(\Omega, \Sigma)$ is an adjoint pair with unit $\eta$ and counit $\varepsilon$. $(ii)$ \ $(\Omega, \triangle_l)$ is a left triangulated structure on $\mcC$ and $(\Sigma, \triangle^r)$ is a right triangulated structure on $\mcC$. $(iii)$ \ For any diagram in $\mcC$ with commutative left square
\[
\xymatrix{
A\ar[r]^f\ar[d]_{\alpha} & B\ar[r]^g\ar[d]_{\beta} &C\ar@{.>}[d]_{\delta} \ar[r]^h & \Sigma(A)\ar[d]^{\varepsilon_{C'}\Sigma(\alpha)}    \\
\Omega(C') \ar[r]^{f'} & \Omega(A')\ar[r]^{g'} & B' \ar[r]^{h'} & C' . }
\]
where the upper row is in $\triangle^r$ and the lower row is in $\triangle_l$, there exists a morphism $\delta$ making the diagram commutative. $(iv)$ \ For any diagram in $\mcC$ with commutative right square
\[
\xymatrix{
A\ar[r]^f\ar[d]_{\Omega(\alpha)\eta_{A}} & B\ar[r]^g\ar@{.>}[d]_{\delta} &C\ar[d]_{\beta} \ar[r]^h & \Sigma(A)\ar[d]_{\alpha}    \\
\Omega(C') \ar[r]^{f'} & \Omega(A')\ar[r]^{g'} & B' \ar[r]^{h'} & C' . }
\]
where the upper row is in $\triangle^r$ and the lower row is in $\triangle_l$, there exists a morphism $\delta$ making the diagram commutative.

$\mcC$ is called a {\it pretriangulated category} if there is a pretriangulated structure on $\mcC$, \cite[Definition 4.9]{Beligiannis01}, \cite[Definition 6.5.1]{Hovey99}.
\vskip10pt
\begin{corollary} \ $(i)$ \ The left triangulated structure $(\Omega, \triangle_l)$ and right triangulated structure $(\Sigma, \triangle^r)$ is a pretrianguated structure on the homotopy category $\Ho(\A)$.

\vskip5pt
$(ii)$ \ The left triangulated structure $(\Omega, \triangle_l)$ and right triangulated structure $(\Sigma, \triangle^r)$ induce a pretrianguated structure on the stable category $\A_{cf}/\omega$.
\end{corollary}

\begin{proof} \ Since we have equivalences $\Ho(\A_f)\simeq \Ho(\A)\simeq \Ho(\A_c)$, by Theorem 4.12, $(\Omega, \triangle_l)$ and $(\Sigma, \triangle^r)$ give the homotopy category $\Ho(\A)$ left and right triangulated structure. By Lemma 4.13, we only need to verify $(iii)$ and $(iv)$ in the above definition. For $(iii)$, without loss of generality,
we only consider the following diagram
\[
\xymatrix{
A\ar[r]^{\gamma(f)}\ar[d]_{\gamma(\alpha)} & B\ar[r]^{\gamma(\mu^{f})}\ar[d]_{\gamma(\beta)} &{\rm PO}(f) \ar[r]^{\gamma(\pi^f)} & \Sigma(A)\ar[d]^{\varepsilon_{C}\Sigma Q(\gamma(\alpha))}    \\
\Omega(C) \ar[r]^{\gamma(\zeta_{g})} & {\rm PB}(g)\ar[r]^{\gamma(\eta_g)} & D \ar[r]^{\gamma(g)} & C . }
\]
with $A, B, C, D\in \A_{cf}$ (note that in this case any the morphism from $A$ to $\Omega(C)$ is of the form $\gamma(\alpha)$) such that the left square commutative. Let $\varphi$ be the adjunction isomorphism. Then $\varepsilon_{C}\Sigma Q(\gamma(\alpha))=\varphi^{-1}(\alpha)$, by definition (see the proof of Lemma 4.13) this is $-\gamma(\kappa^\alpha)$ as constructed by the commutative diagram
\[
\xymatrix{
0\ar[r] & A\ar[r]^{\nu^A}\ar[d]_{\alpha} & W^A\ar[d]_{y^\alpha} \ar[r]^{\pi^A} & \Sigma(A)\ar[d]^{\kappa^{\alpha}} \ar[r]& 0   \\
0 \ar[r] & \Omega(C)\ar[r]^{\iota_C} & W_C \ar[r]^{p_{_{C}}} & C \ar[r] &0. }
\]
Since $\gamma(\beta f)=\gamma(\zeta_g\alpha)$, by $(iii)$ of Lemma 4.9, we know that $\beta f-\zeta_g\alpha$ factors through $\omega$
\[
\xymatrix{ A\ar[r]^{\beta f-\zeta_g\alpha \ \ \ } \ar[rd]_{t}& {\rm PB}(g) \\
& W \ar[u]_{s} }
\]
Since $W\in \omega$, there is a morphism $l: W^A\to W$ such that $l\nu^A=t$.  Denote the cobase change of $f$ along $\nu^A$ by $m^f$.  Since we have $\eta_g \beta f=\eta_g(st+\zeta_g \alpha)=\eta_gst=\eta_gsl\nu^A$, by the pushout property, there is a morphism $\delta: {\rm PO}(f)\to D $ such that $\delta \mu^f=\eta_g\beta$ and $\delta m^f=\eta_g sl$. Similarly, since $\theta_g\beta f=\theta_gst+\theta_g\zeta_g\alpha=\theta_g sl\nu^A +\iota_C\alpha=(y^\alpha +\theta_g sl)\nu^A$, there exists a morphism $z: {\rm PO}(f)\to W_C$ such that $z\mu^f=\theta_g\beta$ and $zm^f=y^\alpha +\theta_gsl$. Then $(p_{_C}z-g\delta -\kappa^\alpha \pi^f)m^f=0=(p_{_C}z-g\delta -\kappa^\alpha \pi^f)\mu^f$ and thus $p_{_C}z-g\delta -\kappa^\alpha \pi^f=0$ by the uniqueness of the pushout property. So we have $\ul{g}\ul{\delta}=-\ul{\kappa}^\alpha\ul{\pi}^f$ and then $\gamma(g)\gamma(\delta)=-\gamma(\kappa^\alpha)\gamma(\pi^f)$ by Lemma 4.9. The morphism $\gamma(\delta): {\rm PO}(f)\to D$ is the desired fuller. Dually we can prove that the assertion $(iv)$ in the above definition holds.

$(ii)$ \ Let $Q(\triangle_l)$ be the class of left triangles of the form $Q\Omega(C)\stackrel{Q(\ul{h})}\To Q(E)\stackrel{Q(\ul{g})}\To Q(D)\stackrel{Q(\ul{f})}\To Q(C)$. Where $\Omega(C)\stackrel{\gamma(h)}\To E \stackrel{\gamma(g)}\To D \stackrel{\gamma(f)} \To C$ is a distinguished left triangle in $\triangle_l$. Then $(Q\Omega, Q(\triangle_l))$ is a left triangulated structure on $\A_{cf}/\omega$ (note that $Q:\Ho(\A_f)\to \A_{cf}/\omega$ is an equivalent). Dually, let $R(\triangle^r)$ be the class of right triangles of the form $R(A)\stackrel{R(\ul{f})}\To R(B)\stackrel{R(\ul{g})}\To R(C)\stackrel{R(\ul{h})}\To R\Sigma(A)$. Where $A\stackrel{\gamma(f)}\to B \stackrel{\gamma(g)}\To C \stackrel{\gamma(h)} \To \Sigma (A)$ is a distinguished right triangle in $\triangle^l$. Then $(R\Sigma, R(\triangle^r))$ is a right triangulated structure on $\A_{cf}/\omega$ (note that $R:\Ho(\A_c)\to \A_{cf}/\omega$ is an equivalent). Note that the stable category $\A_{cf}/\omega$ are subcategories of $\Ho(\A_f)$ and $\Ho(\A_c)$, we have that $\triangle_l$ contains $Q(\triangle_l)$ and $\triangle^r$ contains $R(\triangle^r)$ respectively. Then by $(i)$, the quadruple $(R\Sigma, Q\Omega, Q(\triangle_l), R(\triangle^r)$ is a pretriangulated structure on $\A_{cf}/\omega$.
\end{proof}

\section{Appendix}
In this appendix we will show that for an additive category $\mcC$ and a contravariantly finite additive subcategory $\X$, if $\X$ is closed under direct summands and every $\X$-epic has a kernel, then $\mcC$ becomes an additive fibration category with fibration structure induced by $\X$. Moreover the stable category $\mcC/\X$ is just the homotopy category of the fibration category $\mcC$. So it has a left triangulated structure as has been shown by K. S. Brown in \cite{Brown74}. This left triangulated structure coincides with the one constructed in Section 3.2.
\subsection{The homotopy categories of additive (co)fibration categories }
The notion of fibration categories was first introduced and studied by K. S. Brown under the name ``categories of fibrant objects" in \cite{Brown74}. The dual notion of fibration categories is cofibration categories. For more details we refer \cite{Radulescu-Banu06} and \cite{Schwede12}. For simplicity, we only state the case of fibration categories and omit the dual one.
\vskip5pt
\begin{definition} \ An {\it additive fibrantion category} $(\mcC, \mathcal{W}e, \mathcal{F}ib)$ consists of an additive category $\mcC$ and two classes of morphisms, the {\it fibrations} $\mathcal{F}ib$ respectively the {\it weak equivalences} $\mathcal{W}e$, that satisfy the following axioms $({\rm F}1)-({\rm F}4)$.
\vskip5pt
$({\rm F}1)$ \ All isomorphisms are fibrations and weak equivalences. Fibrations are closed under composite. Every morphism to a terminal object is a fibration.
\vskip5pt
$({\rm F}2)$ \ (Two out of three axiom) Given two composable morphisms $f$ and $g$ in $\mcC$, then if two of the three morphisms $f$, $g$ and $gf$ are weak equivalences, so is the third.

\vskip5pt
$({\rm F}3)$ \ (Pullback axiom) Given a fibration $p: B\to A$ and any morphism $f:C \to A$, there exists a pushout square
 \[
\xymatrix{
D \ar[r] \ar[d]_q& B \ar[d]^p \\
C \ar[r]_{f} & A }
\]
in $\mcC$ and the morphism $q$ is a fibration. If additionally $p$ is a weak equivalence, then so is $q$.
\vskip5pt
$({\rm F}4)$ \ (Factorization axiom) Every morphism in $\mcC$ can be factored as the composite of a weak equivalence followed by a fibration.
\end{definition}

\begin{remark} \ $(i)$ \ Dually we can define the notion of an {\it additive cofibration category} $(\mcC, \mathcal{W}e, \mcC of)$. The morphism in the class $\mcC of$ is called {\it cofibration}.
\vskip5pt
$(ii)$ \ The notions of fibration and cofibration categories can be defined on general categories and are a substantial generalization of Quillen's closed model category \cite{Quillen67}. From a closed model category one obtains a fibration category by restricting to the full sucategory of cofibrant objects and a cofibration category by restricting to the full subcategory of fibrant objects. The further relevant reference on fibration and cofibration category are \cite{Brown74}, \cite{Schwede12} and \cite{Radulescu-Banu06}.

\vskip5pt
$(iii)$ \ The pair $(\mathcal{W}e, \mathcal{F}ib )$ (respectively $(\mathcal{W}e, \mcC of)$) is called a {\it fibration structure} (respectively {\it cofibration structure}) if $(\mcC, \mathcal{W}e, \mathcal{F}ib )$ (respectively $(\mcC, \mathcal{W}e, \mcC of)$)is a fibration (cofibration) category.
\end{remark}

\vskip10pt
\begin{definition} \ ( \cite{Gabriel/Zisman67})  Let $\mcC$ be a category, and $\mathcal{W}$ a class of morphisms of $\mcC$. A functor $\gamma:\mcC\to \mathcal{D}$ is said to be a {\it localization of $\mcC$ with respect to $\mathcal{W}$} if

$(i)$\ $\gamma(f)$ is an isomorphism for each $f\in \mathcal{W}$;

$(ii)$ whenever $F: \mcC\to \mathcal{D}'$ is a functor carrying elements of $\mathcal{W}$ into isomorphisms, there exists a unique functor $G:\mcD\to \mcD'$ such that $G\gamma=F$.
\end{definition}

The homotopy category of an additive fibration category $\mcC$ is any localization of $\mcC$ at the class of weak equivalences. That is a pair $(\Ho(\mcC), \gamma)$ consists of a category $\Ho(\mcC)$ with the same objects as $\mcC$ and a localization $\gamma: \mcC\to \Ho(\mcC)$ that is identity on objects. Next we recall the description of $\Ho(\mcC)$ in \cite{Brown74}.

A {\it path object} for an object $A$ in a fibration category is an object $A^I$ of $\mcC$ together with a factorization of the diagonal map $A\stackrel{(1_A, 1_A)}\To A\prod A $:
$$A\stackrel{q}\to A^I\stackrel{(p_0,p_1)}\To A\prod A $$
 where $q$ is a weak equivalence and $(p_o,p_1)$ a fibration such that $p_0q=p_1q=1_A$. Note that every object has a path object and the morphisms $p_0,p_1$ are {\it acyclic fibrations}. Thus $\gamma(p_0)=\gamma(p_1)$ in $\Ho(\mcC)$ since $\gamma(q)$ is an isomorphism.

Two morphisms $f,g:A\to B$ in a fibration category are called {\it homotopic} if there exists a path object $B^I$ for $B$ and a morphism $H:A\to B^I$ (called the {\it homotopy}) such that $f=p_0H$ and $g=p_1H$. If $f$ and $g$ are homotopic we denote by $f\stackrel{h}\sim g$. Homotopic morphisms become equal in the homotopy category. In fact if $f\stackrel{h}\sim g$ via $H$, then $\gamma(f)=\gamma(p_0)\gamma(H)=\gamma(p_1)\gamma(H)=\gamma(g)$.

Two morphisms $f,g:A\to B$ in a fibration category $\mcC$ are {\it equivalent} if there is an acyclic fibration $t:A'\to A$ such that $ft\stackrel{h}\sim gt$. We will write $f\sim g$. This is an equivalence relation. Define a category $\pi\mcC$ to be the category with the same object as $\mcC$ and with $\Hom_{\pi\mcC}(A,B)$ equal to the quotient of $\Hom_\mcC(A,B)$ by the equivalence relation $\sim$. Then the weak equivalence class in $\pi\mcC$ admits a calculus of ``right fractions", and the homotopy category $\Ho(\mcC)$ is obtained from $\pi\mcC$ by inverting the weak equivalence, see \cite[Proposition 2]{Brown74} and \cite[Chapter I.2-I.3]{Gabriel/Zisman67}.

\vskip10pt
Part $(i)$ and $(ii)$ of the following theorem are the Theorem 1 and Remarks 2 of \cite{Brown74}. Part $(iii)$ is the dual statement of Corollary I.3.3.2 of \cite{Gabriel/Zisman67}.

\vskip10pt
\begin{theorem} $($\cite[Theorem 1, Remarks 2]{Brown74}$)$ \ Let $\mcC$ be a fibration category and $\gamma:\mathcal{M}\to \Ho(\mcC)$ the localization functor, then
\vskip5pt
$(i)$ \ The morphisms in $\Ho(\mcC)$ are precisely of the form $\gamma(f)\gamma(t)^{-1}$, where $f, t$ are morphisms in $\pi\mcC$ and $t$ is an acyclic fibration.
\vskip5pt
$(ii)$ \ If $f,g:A\rightrightarrows B$ are morphisms in $\mcC$ then $\gamma(f)=\gamma(g)$ if and only if there is an acyclic fibration $t:A'\to A$ such that $ft\stackrel{h}\sim gt$.
\vskip5pt
$(iii)$ If further $\mcC$ is an additive category,  so is $\Ho(\mcC)$.
\end{theorem}

\subsection{The induced (co)fibration structures}
Now let $\mcC$ be an additive category and $\X$ a full additive subcategory. Assume that $\X$ is closed under direct summands. We recall the following notations which were introduced in \cite[Definition 4.7]{Beligiannis01}.
\begin{definition} \ $\X$ is called {\it left (right) homotopic} if the following conditions hold:
\vskip5pt
$(i)$ \ $\X$ is a contravariantly (covariantly) finite in $\mcC$.
\vskip5pt
$(ii)$ \ Any $\X$-epic ($\X$-monic) has a kernel (cokernel) in $\mcC$.
\end{definition}

For an additive category $\mcC$ and its additive full subcategory $\X$ which is closed under direct summands. Define classes of morphism as $\mathcal{F}ib_\X$ the class of $\X$-epics, $\mathcal{C}of_\X$ the class of $\X$-monics and  $\mathcal{W}e_\X$ the class of stable equivalences (i.e. morphisms which are isomorphisms in the stable category $\mcC/\X$).
We call a fibration (respectively cofibration) structure $(\mathcal{W}e, \mathcal{F}ib)$ (respectively $(\mathcal{W}e, \mathcal{C}of)$) on $\mcC$ {\it induced} if there is some full additive subcategory $\X$ which is closed under direct summands. such that $\mathcal{W}e=\mathcal{W}e_\X$ and $\mathcal{F}ib=\mathcal{F}ib_\X$ ($\mathcal{C}of=\mcC of_\X$).

\vskip10pt Motivated by Theorem 4.5 of \cite{Beligiannis01} on the relationship between the functorially finite subcategories and the induced model structure for an additive category with split idempotents. We have the following result. Note that we do not assume that the additive category has split idempotents.
\vskip10pt
\begin{theorem} \ $(i)$ \ There is a bijection:
\[
\left\{\begin{gathered}\text{induced fibration }\\ \text{structures of $\mcC$}
\end{gathered}\;
\right\} \xymatrix@C=1pc{ \ar[r]^-{\sim} & }\left\{
\begin{gathered}
  \text{left homotopic}\\\text{ subcategories of $\mcC$
    }
\end{gathered}
\right\}\,
\]
which sends an induced fibration structure $(\mathcal{W}e_\X,\mathcal{F}ib_\X)$  to $\X$. The inverse map sends a left homotopic subcategory $\X$ to $(\mathcal{W}e_\X, \mathcal{F}ib_\X)$.
\vskip5pt
$(ii)$ Dually, there is a bijection:
\[
\left\{\begin{gathered}\text{induced cofibration }\\ \text{structures of $\mcC$}
\end{gathered}\;
\right\} \xymatrix@C=1pc{ \ar[r]^-{\sim} & }\left\{
\begin{gathered}
  \text{right homotopic}\\\text{ subcategories of $\mcC$
    }
\end{gathered}
\right\}\,
\]
which sends an induced cofibration structure $(\mathcal{W}e_\X,\mathcal{C}of_\X)$ to $\X$. The inverse map sends a right homotopic subcategory $\X$ to $(\mathcal{W}e_\X, \mathcal{C}of_\X)$.
\end{theorem}

\begin{proof} \ The second statement follows from $(i)$ by duality. So we only prove the first. Given an induced fibration structure $(\mathcal{W}e, \mathcal{F}ib)$ on $\mcC$, assume that it is induced by $\X$. By the pullback axiom we know that any fibration, i.e. $\X$-epic morphism has a kernel. For any $A\in \mcC$, take a nonzero morphism $f: X'\to A$ with $X'\in \X$. Factoring $f$ as the composite $X'\stackrel{i}\to X\stackrel{p}\to A$ such that $i$ is a weak equivalence and $p$ a fibration. By definition $\ul{i}$ is an isomorphism in $\mcC/\X$ and $p$ is $\X$-epic. Thus $\ul{X}\cong \ul{X'}\cong 0$. So the identity $1_X$ can be factored as a composite $X\stackrel{t}\to X''\stackrel{s}\to X$. Since $s$ is $\X$-epic, it has a kernel and then $X$ is a direct summand of $\X''$ since $s$ is split. By assumption, $\X$ is closed under direct summands, we know that $X\in \X$ and then $p$ is a right $\X$-approximation of $A$. That is to say $\X$ is contravariantly finite in $\mcC$.

Conversely, given a left homotopic subcategory $\X$ of $\mcC$. Let $\mathcal{W}e_\X,\mathcal{F}ib_\X$ be defined as above. We need to verify the axioms of Definition 5.1 one by one. The axioms $({\rm F}1)$ and $({\rm F}2)$ of Definition 5.1 are obviously by the construction of $\mathcal{W}e_\X$ and $\mathcal{F}ib_\X$. For any morphism $f: A\to B$, take a right $\X$-approximation $p_{_B}: X_B\to B$, then $f$ can be written as $A\stackrel{\left(\begin{smallmatrix}
1 \\
0
\end{smallmatrix}\right)}\to A\oplus X_B\stackrel{(f, p_{_B})}\to B$ with $(f, p_{_B})\in \mathcal{F}ib_\X$ and $\left(\begin{smallmatrix}
1 \\
 0
\end{smallmatrix}\right) \in \mathcal{W}e_\X$. Thus the factorization axiom $({\rm F}4)$ holds. For any diagram
\[
\xymatrix{
& B \ar[d]^p \\
C \ar[r]_{f} & A }
\]
with $p$ a fibration. Since $p$ is $\X$-epic, we know that $(f, p): C\oplus B\to A$ is also an $\X$-epic, thus it has a kernel $D$ by assumption. $D$ is the pullback of $f$ and $p$, i.e. there is a pullback square
\[
\xymatrix{
D \ar[r] \ar[d]_q& B \ar[d]^p \\
C \ar[r]_{f} & A }
\]
By the pullback universal property we know that $q$ is also an $\X$-epic and then a fibration. If in addition $p$ is a weak equivalence, i.e. $p$ is an acyclic fibration. Then $p$ is split epimorphism with kernel in $\X$ by Lemma 4.4 (2) of \cite{Beligiannis01}. Then the base change $q$ is also a split epimorphism with kernel in $\X$, thus it is also an acyclic fibration. So we proved the pullback axiom $({\rm F}3)$.
\end{proof}

Let $\mcC$ be an additive fibration category with induced fibration structure by $\X$. Then the homotopy category $\Ho(\mcC)$ is the stable category $\mcC/\X$. In fact, any morphism in $\Ho(\mcC)$ is of the form $\gamma(f)\gamma(t)^{-1}$ with $t$ an acyclic fibration by $(i)$ of Theorem 5.4. Thus $t$ is a split epimorphism with kernel in $\X$ by Lemma 4.4 (2) of \cite{Beligiannis01}, i.e. $t$ has a right inverse. So any morphism of $\Ho(\mcC)$ is of the form $\gamma(f)$ with $f$ in $\mcC$. Thus $\Hom_{\Ho(\mcC)}(A,B)\cong \Hom_{\pi\mcC}(A,B)$ (see Section 5.1). Similarly we know that $f\sim g$ in $\mcC$ if and only if $f\stackrel{h}\sim g$. But $f\stackrel{h}\sim g$ is equivalent to $f$ and $g$ are stable equivalent. Thus $\Ho(\mcC)\cong \mcC/\X$. For $A\in \Ho(\mcC)$, the suspension of $A$ is constructed by K. S. Brown in \cite{Brown74} as the kernel of a fibration $ A^I\to A\oplus A$ which is part of a path object of $A$. But in this case, any path object of $A$ is of the form $A\stackrel{\left(\begin{smallmatrix}
1  \\
0
\end{smallmatrix}\right)}\to A\oplus X_A\stackrel{\left(\begin{smallmatrix}
1 & p_{_A} \\
1 & 0
\end{smallmatrix}\right)} \to A\oplus A$, where $p_A$ is a right $\X$-approximation of $A$. So the suspension of $A$ is the kernel $p_{_A}$ which is the same as constructed in Section 3.2. The distinguished left triangles constructed by Brown is the same as the induced left triangles as we have defined in Section 3.2.

\end{document}